\documentclass[12pt]{amsart}
\usepackage{amssymb,amsfonts,latexsym,amscd}

\usepackage[all]{xy}
\usepackage{verbatim}
\usepackage{hyperref}
\usepackage{mathrsfs}
\usepackage{mathtools}
\usepackage{amssymb}
\usepackage{stmaryrd}
\usepackage{tikz-cd}

\usepackage{xcolor}

\newtheorem*{rep@theorem}{\Rep_k@title}
\newcommand{\newreptheorem}[2]{%
	\newenvironment{rep#1}[1]{%
		\def\Rep_k@title{#2 \ref{##1}}%
		\begin{rep@theorem}}%
		{\end{rep@theorem}}}
\makeatother
\newreptheorem{theorem}{Theorem}

\numberwithin{equation}{section}

\newtheorem{theorem}{Theorem}[section]
\newtheorem{lemma}[theorem]{Lemma}
\newtheorem{proposition}[theorem]{Proposition}
\newtheorem{corollary}[theorem]{Corollary}
\theoremstyle{definition}
\newtheorem{definition}[theorem]{Definition}

\newtheorem{example}[theorem]{Example}

\newtheorem{remark}[theorem]{Remark}

\newcommand{\id}{{\rm id}}

\newcommand{\Ker}{\text{Ker\,}}

\newcommand{\Fun}{{\rm Fun}}
\newcommand{\Irr}{\text{\rm Irr}}

\newcommand{\FPdim}{\text{\rm FPdim}}

\newcommand{\Hom}{{\rm Hom}}

\newcommand{\Rep}{{\rm Rep}}

\newcommand{\Vect}{{\rm Vec}}

\renewcommand{\O}{\mathscr{O}}

\newcommand{\h}{\mathfrak{h}}

\newcommand{\ot}{\otimes}

\newcommand{\ben}{\begin{enumerate}}
\newcommand{\een}{\end{enumerate}}

\newcommand{\Lie}{{\rm Lie}}

\newcommand{\C}{{\mathscr C}}

\newcommand{\M}{{\mathscr M}}

\newcommand{\Z}{{\mathscr{Z}}}

\begin{document}

\title[On finite group scheme-theoretical categories, I]{On finite group scheme-theoretical categories, I}

\author{Shlomo Gelaki}
\address{Department of Mathematics,
Iowa State University, Ames, IA 50100, USA.}
\email{gelaki@iastate.edu}

\author{Guillermo Sanmarco}
\address{Department of Mathematics,
Iowa State University, Ames, IA 50100, USA.} 
\email{sanmarco@uw.edu}

\date{\today}


\dedicatory{In memory of Earl J. Taft}

\begin{abstract}
Let $\C(G,H,\psi)$ be a finite group scheme-theoretical category over an algebraically closed field of characteristic $p\ge 0$ \cite{G}. For any indecomposable exact $\C(G,H,\psi)$-module category, we classify its simple objects and provide an expression for their projective covers, in terms of double cosets and projective representations of certain closed subgroup schemes of $G$, which upgrades a result by Ostrik \cite{O} for group-theoretical fusion categories. As a byproduct, we describe the simples and indecomposable projectives of $\C(G,H,\psi)$, and parametrize the Brauer-Picard group of ${\rm Coh}(G)$ for any finite connected group scheme $G$. Finally, we apply our results to describe the blocks of the center of ${\rm Coh}(G)$.
\end{abstract}

\maketitle

\tableofcontents

\section{Introduction}

The main purpose of this paper is to study the structure of indecomposable exact module categories over finite group scheme-theoretical categories \cite{G} with focus on classifying their simple objects and describing their projective covers. Our study in particular sheds some light on the structure of group scheme-theoretical categories themselves, and applies to the representation category of the twisted double of a finite group scheme, for which a more precise description is obtained. Throughout, we work over an algebraically closed field $k$ of characteristic $p\ge 0$. All schemes in this paper are assumed to be finite over $k$, unless otherwise stated. 

Group-theoretical fusion categories \cite{O} play a fundamental role in the theory of fusion categories in characteristic $0$. These categories can be explicitly constructed from finite group data, which makes them treatable enough to serve as a common testing ground for several questions within the theory \cite{EKW,G-e,Na}. At the same time, this class of fusion categories encompasses enough to play a part in general classification efforts as explained for example in \cite[\S 9.13]{EGNO}. 

From an abstract point of view, group-theoretical fusion categories are characterized as those that are dual to a pointed fusion category with respect to an indecomposable module category.
For a concrete construction, one starts with a finite group $G$ and a normalized $3$-cocycle $\omega \in Z^3(G,k^{\times})$, and considers the pointed fusion category $\Vect_G^\omega$ of finite dimensional $G$-graded vector spaces with associativity constraint given by $\omega$. By \cite{Na,O}, indecomposable module categories over $\Vect_G^\omega$ are classified, up to equivalence, by certain conjugacy classes of pairs $(H, \psi)$, where $H\subset G$ is a subgroup and $\psi\in C^2(H,k^\times)$ is a normalized $2$-cochain such that $d \psi = \omega_{\vert H}$. Given such a pair $(H,\psi)$, the twisted group algebra $k^\psi[H]$ becomes an algebra in $\Vect_G^\omega$, and the group-theoretical category $\C(G, \omega, H, \psi)$ associated to this data is defined as the category of $k^\psi[H]$-bimodules in $\Vect_G^\omega$.

By \cite{O} (see also \cite{M+}), equivalence classes of indecomposable module categories over $\C (G, \omega, H, \psi)$ are classified by certain conjugacy classes of pairs $(K, \eta)$, where $K\subset G$ is a subgroup and $\eta\in C^2(K,k^\times)$ such that $d \eta = \omega_{\vert K}$. Ostrik also provided a formula that computes the rank of the module category associated to $(K,\eta)$ in terms of $(H,K)$-double cosets in $G$ and certain subgroups built from them. Ostrik developed this theory mainly to classify indecomposable module categories, particularly fiber functors, over the group-theoretical modular category $\mathscr{Z}\left(\Vect_G^\omega\right)$.

The problem of classifying indecomposable exact module categories over the finite tensor category $\Rep_k(G)$ of finite dimensional representations over $k$ of a finite group $G$ had been previously solved by Ostrik \cite{O2} when $p=0$, and then extended by Etingof-Ostrik \cite{EO} to  $p>0$. However, the dual of $\Rep_k(G)$ with respect to the module category given by the standard forgetful functor $\Rep_k(G) \to \Vect$ coincides with the fusion category $\Vect_G$ of finite dimensional $G$-graded $k$-vector spaces. Thus, one can alternatively solve this problem using the classification of module categories over $\Vect_G$.

In a pursuit of a classification of indecomposable exact module categories over $\Rep_k(G)$ for an affine $k$-group scheme $G$, in \cite{G} the first author generalized  Ostrik's construction \cite{O} to affine group schemes in  any characteristic, but in this paper we focus on the {\em finite} case. Given a finite group scheme $G$ and a normalized $3$-cocycle $\omega \in Z^3(G,\mathbb G_m)$ (equivalently, a \emph{Drinfeld associator} $\omega\in\O(G)^{\ot 3}$ for the coordinate algebra), one first defines the category ${\rm Coh}(G, \omega)$ of  sheaves on $G$ with tensor product given by convolution, and associativity by the action of $\omega$. Now, for any closed subgroup scheme $H\subset G$ and  normalized $2$-cochain $\psi\in C^2(H,\mathbb G_m)$ such that $d \psi=\iota_H^{\sharp \ot 3}(\omega)$, there is a notion of a $(H,\psi)$-equivariant  sheaf on $(G,\omega)$. Moreover, the category $\M(H,\psi)$ they form has the structure of an indecomposable exact module category over ${\rm Coh}(G,\omega)$ under convolution of sheaves.

\begin{theorem}{\cite[Theorem 5.3]{G}}
The assignment $(H,\psi)\mapsto \M(H,\psi)$ determines a bijection between conjugacy classes of pairs $(H,\psi)$, where $H\subset G$ is a closed subgroup scheme and $\psi\in C^2(H,\mathbb G_m)$ is normalized such that $d \psi=\iota_H^{\sharp \ot 3}(\omega)$, and equivalence classes of (left) indecomposable exact module categories over ${\rm Coh}(G, \omega)$. \qed
\end{theorem}

By \cite[Theorem 3.31]{EO}, since ${\rm Coh}(G)$ is dual to $\Rep_k(G)$ with respect to the standard forgetful functor, \cite[Theorem 5.3]{G} also classifies indecomposable exact module categories over $\Rep_k(G)$: as in the \'etale case \cite[Prop. 4.1]{EO}, these are all of the form $\Rep_k(H,\psi)$ for a closed subgroup scheme $H\subset G$, and a normalized $2$-cocycle $\psi \in Z^2(H,\mathbb G_m)$.  

Going back to the data above, one defines the \emph{group scheme-theoretical} category $\C(G, \omega,H, \psi)$ as the dual of ${\rm Coh}(G, \omega)$ with respect to $\M(H,\psi)$. The first author also classified indecomposable exact module categories over the finite tensor category $\C(G, \omega,H, \psi)$. 

\begin{theorem}\cite[Theorem 5.7]{G}
The assignment 
$$\mathscr{M}(K,\eta)\mapsto \Fun_{{\rm Coh}(G, \omega)}\left(\mathscr{M}(H,\psi),\mathscr{M}(K,\eta)\right)$$
determines an equivalence between the $2$-category of indecomposable exact module categories over ${\rm Coh}(G,\omega)$ and the $2$-category of indecomposable exact module categories over $\C(G,\omega, H,\psi)$. \qed
\end{theorem}

In contrast with the fusion case in characteristic $0$, group scheme-theoretical categories and their module categories are rarely semisimple if $p>0$, and a plethora of questions arises. In this paper we focus on the case $\omega=1$, and delegate to future work the treatment in the presence of a general $3$-cocycle. So, fix a finite group scheme $G$, a closed subgroup scheme $H\subset G$, and a normalized $\psi \in Z^2(H, \mathbb G_m)$ (equivalently, a \emph{Drinfeld twist} $\psi \in \O(H)^{\ot 2}$). 

We begin our work in Section \S\ref{sec:Preliminaries} with preliminaries about finite group schemes and module categories over their categories of sheaves. Our main result here is Theorem \ref{helpful2}, where we adapt \cite[Theorem 1(B), p.112]{Mum} to show that the abelian category $\M(H,\psi)$ of $(H,\psi)$-equivariant sheaves on $G$ is equivalent to the category ${\rm Coh} (G/H)$ of  sheaves over the finite quotient scheme $G/H$. 

In Section \S\ref{S:Biequivariant sheaves on group schemes} we consider an embedding $\partial\colon A\xrightarrow{1:1} B$ of finite group schemes, a twist $\Psi\in Z^2(B,\mathbb{G}_m)$, and its restriction $\xi:=\partial^{\sharp \ot 2}(\Psi)$ to $A$.

\begin{reptheorem}{modrep0-1}
There are equivalences of abelian categories
\begin{equation*}
	\begin{tikzcd}[column sep=4em]
		\mathrm{Rep}_k(A,\xi^{-1}) 
		\arrow[r] 
		\arrow[rr, "{\rm Ind}_{(A,\xi^{-1})}^{(B,\Psi)}", bend left=7, shift left=0.5ex]
		&
		\mathrm{Coh}^{(A\times B,\xi^{-1}\times \Psi)}(B) 
		\arrow[l] 
		\arrow[r]
		&
		\mathrm{Coh}^{(B,\Psi)}(A\backslash B)
		\arrow[l]
		\arrow[ll, "{\rm Res}_{(A,\xi^{-1})}^{(B,\Psi)}", bend left=7, shift left=0.5ex]
	\end{tikzcd}
\end{equation*}
Moreover, ${\rm Ind}_{(A,\xi^{-1})}^{(B,\Psi)}$ and ${\rm Res}_{(A,\xi^{-1})}^{(B,\Psi)}$ are mutually inverse, and can be described  explicitly. \qed
\end{reptheorem}

Using this, we get yet another abelian equivalence that will be used later on for a different purpose, see Section \S\ref{S:Biequivariant sheaves on group schemes} for precise statements. 

\begin{reptheorem}{modrep00}
We have an equivalence of abelian categories   
$${\rm F}\colon \mathrm{Rep}_k(A,\xi^{-1})\xrightarrow{\cong}{\rm Coh}^{(B,\Psi)}(A\backslash B),\,\,\,V\mapsto \O(A\backslash B)\ot_k V.$$
\end{reptheorem}

In Section \S\ref{sec:doublecosetsandbiequivariant} we discuss double cosets in finite group schemes, and biequivariant sheaves. 

In Section \S\ref{sec:Module categories over GCTC} we apply Theorems \ref{modrep0-1}, \ref{modrep00} (using Section \S\ref{sec:doublecosetsandbiequivariant}) to describe indecomposable exact module categories over a group scheme-theoretical category $\C:=\C\left(G,H,\psi\right)$, which are classified by \cite[Theorem 5.7]{G} as recalled above. As a first step, we show in Lemma \ref{newsimple}
that there is an equivalence of abelian categories
$$\M\coloneqq{\rm Coh}^{\left(H\times K,\psi^{-1}\times \eta\right)}(G)\cong \Fun_{{\rm Coh}(G)}\left(\mathscr{M}(H,\psi),\mathscr{M}(K,\eta)\right),$$
where $H\times K$ acts on $G$ on the right via $(g,a,b)\mapsto a^{-1}gb$. 
Fix a closed point $Z\in Y(k)$ of the quotient finite scheme $Y:=G/(H\times K)$ (see \S\ref{sec:quotmorphind}), and let  $\M_{Z}\subset \M$
denote the full abelian subcategory consisting of all objects annihilated by the defining ideal $\mathscr{I}(Z)\subset \mathscr{O}(G)$ of $Z$. Our next result classifies simples in $\M$; we interpret it as a generalization of \cite[Prop. 3.2, Remark 3.3(i)]{O}.

\begin{reptheorem}{simmodrep}
	Let $\M={\rm Coh}^{\left(H\times K,\psi^{-1}\times \eta\right)}(G)$ be as above.\begin{enumerate}
		\item
		For any closed point $Z\in Y(k)$ with representative closed point $g\in Z(k)$, we have an equivalence of abelian categories
		$$\mathbf{Ind}_{Z}:\mathrm{Rep}_k(L^g,\xi_g^{-1})\xrightarrow{\cong}\M_{Z},$$
		where $L^g=H\cap gKg^{-1}$ and $\xi_g \in Z^2(L^g, \mathbb G_m)$ is defined in (\ref{xig0}).
		\item
		There is a bijection between equivalence classes of pairs $(Z,V)$, where $Z\in Y(k)$ with representative closed point $g\in Z(k)$, and $V\in \mathrm{Rep}_k(L^g,\xi_g^{-1})$ is simple, and simple objects of $\mathscr{M}$, assigning $(Z,V)$ to $\mathbf{Ind}_{Z}(V)$.
		\item
		We have a direct sum decomposition of abelian categories
		$$\M=\bigoplus_{Z\in Y(k)}\overline{\M_{Z}},$$
		where $\overline{\M_{Z}}\subset \M$ denotes the Serre closure of $\M_{Z}$ inside $\M$.
	\end{enumerate}  
\end{reptheorem}
This result relies on the equivalence from Theorem \ref{modrep0-1}. On the other hand, Theorem \ref{modrep00} is a better tool to compute projective covers.
\begin{reptheorem}{prsimmodrep}
Let $\M={\rm Coh}^{\left(H\times K,\psi^{-1}\times \eta\right)}(G)$ be as above.
\begin{enumerate}
		\item
		For any closed point $Z\in Y(k)$ with representative closed point $g\in Z(k)$, we have an equivalence of abelian categories
		$$\mathbf{F}_{Z}:\mathrm{Rep}_k(L^g,\xi_g^{-1})\xrightarrow{\cong}\M_{Z}.$$
		\item
		For any simple $V\in \mathrm{Rep}_k(L^g,\xi_g^{-1})$, we have
		$$P_{\M}\left(\mathbf{F}_{Z}(V)\right)\cong \O(G^{\circ})\ot \O(Z(k))\ot_k P_{(L^g,\xi_g^{-1})}(V).$$
	\end{enumerate}  
\end{reptheorem}

We conclude Section \S\ref{sec:Module categories over GCTC} with Corollary \ref{fibfungth}, which provides a classification of fiber functors on $\C\left(G,H,\psi\right)$.

Section \S\ref{S:structure-GSC} uses the previous results to describe the abelian structure of $\C\left(G,H,\psi\right)$. In Theorem \ref{simpobjs} we classify simple objects, compute their Frobenius-Perron dimensions, describe how they relate under dualization and provide a graded decomposition for $\C\left(G,H,\psi\right)$. Then in Theorem \ref{projsimpobjs} we give an alternative parametrization of the simples, use it to compute their projective covers and deduce that $\C\left(G,H,\psi\right)$ is unimodular if so is $\Rep_k(H)$. The description provided for $\C\left(G,H,\psi\right)$ is nicer when $G$ is either \'etale or connected, or $H$ is normal (see \S\ref{sec:The etale case}--\S\ref{sec:The connected case}). 

In Section \S\ref{sec:brauerpicard}, we use \S\ref{sec:The normal case}--\S\ref{sec:The connected case} to classify invertible bimodule categories over ${\rm Coh} (G)$ in the connected case. When $G$ is also abelian, we provide in Theorem \ref{trivialbrpictp} an expression for the multiplication in the Brauer-Picard group of ${\rm Coh} (G)$.

We conclude the paper with Section \S\ref{sec:twisted-double}, in which we focus on the center $\Z(G)\coloneqq \Z(\rm{Coh}(G))$. Since $\Z(G)$ is a group scheme-theoretical category, we can apply the results from \S\ref{S:structure-GSC} to provide in Theorem \ref{simpobjsnew} a description of its simple and projective objects. In particular, we obtain a decomposition of abelian categories
$$\mathscr{Z}(G)=\bigoplus_{C\in {\rm C}(k)}\overline{\mathscr{Z}(G)_{C}},$$
where ${\rm C}$ denotes the finite scheme of conjugacy orbits in $G$, and for each subcategory $\mathscr{Z}(G)_{C}$, we have an explicit abelian equivalence
$$\mathbf{F}_{C}:{\rm Rep}_k(G_C)\xrightarrow{\cong} \mathscr{Z}(G)_{C},\,\,\,V\mapsto \O(C)\ot_k V,$$
where $G_C$ is the stabilizer of $g\in C(k)$.

On the other hand, by \cite[Theorem 3.5]{GNN}, if we set $\mathscr{D}:={\rm Coh}(G)$ and $\mathscr{D}^\circ:={\rm Coh}(G^\circ)$, there is an equivalence of tensor categories
\begin{equation*}
F:\mathscr{Z}(G)\xrightarrow{\cong}\mathscr{Z}_{\mathscr{D}^{\circ}}\left(\mathscr{D}\right)^{G(k)}=\left(\bigoplus_{a\in G(k)}\mathscr{Z}_{\mathscr{D}^{\circ}}\left(\mathscr{D}^{\circ}\boxtimes a\right)\right)^{G(k)}.
\end{equation*}

\begin{reptheorem}{compdeco}
For any $C\in {\rm C}(k)$, the functor $F$ restricts to an equivalence of abelian categories
$$F_C:\overline{\mathscr{Z}(G)_{C}}\xrightarrow{\cong}\bigoplus_{a\in C(k)}\mathscr{Z}_{\mathscr{D}^{\circ}}\left(\mathscr{D}^{\circ}\boxtimes a\right)^{G(k)}.$$
In particular, $F$ restricts to an equivalence of tensor categories $$F_1:\overline{{\rm Rep}_k(G)}\xrightarrow{\cong}\mathscr{Z}\left(G^{\circ}\right)^{G(k)}.$$
\end{reptheorem}

\section{Preliminaries}\label{sec:Preliminaries}
In this section we fix notation and recall some basic results about finite (group) schemes to be used throughout the paper (see \cite{J,W} for more details). We assume familiarity with the theory of finite tensor categories and their modules categories, and refer to \cite{EGNO} for any unexplained notion. Recall that we work over an algebraically closed field $k$ of characteristic $p\ge0$. All schemes considered in this paper are assumed to be {\em finite} over $k$, unless otherwise stated. 

\subsection{Sheaves on finite schemes}\label{sec:coh decomposition}
Let $X$ be a finite scheme, i.e., $\O(X)$ is a finite dimensional commutative algebra. Let ${\rm Coh}(X)$ be the abelian category of  sheaves on $X$, i.e., the category of finite dimensional representations of $\O(X)$. 
For each closed point $x\in X(k)$, let ${\rm Coh}(X)_x\subset {\rm Coh}(X)$ denote the abelian subcategory of sheaves on $X$ supported on $x$. Each simple object of ${\rm Coh} (X)$ is contained in a unique ${\rm Coh} (X)_x$, which contains only one simple (up to isomorphism), so we denote that simple object by $\delta_x$. It is well known that there are no extensions of $\delta_x$ by any other nonisomorphic simple object, i.e., we have a decomposition
\begin{equation*}
{\rm Coh} (X)= \bigoplus_{x \in X(k)}{\rm Coh}(X)_x
\end{equation*}
of abelian categories. The projective cover $P_{x}:=P(\delta_x)$ lies in ${\rm Coh} (X)_x$.

Let $Y$ be a finite scheme, $\varphi:Y\to X$ a scheme morphism, and $\varphi^{\sharp}:\O(X)\to\O(Y)$ the corresponding algebra homomorphism. Recall that $\varphi$ induces a pair $(\varphi^*,\varphi_*)$ of adjoint functors of abelian categories
\begin{equation}\label{invimage}
\varphi^*:{\rm Coh}(X)\to {\rm Coh}(Y),\,\,\,S\mapsto S\ot_{\O(X)}\O(Y),
\end{equation}
where $\O(X)$ acts on $\O(Y)$ via $\varphi^{\sharp}$, and  
\begin{equation}\label{dirimage}
\varphi_*:{\rm Coh}(Y)\to {\rm Coh}(X),\,\,\,T\mapsto T_{\mid \O(X)},
\end{equation}
where $\O(X)$ acts on $T$ via $\varphi^{\sharp}$.

\subsection{Finite group schemes}\label{S:Finite group schemes}
A finite group scheme $G$ is a finite scheme whose coordinate algebra $\O(G)$ is a (finite dimensional commutative) Hopf algebra (see, e.g., \cite{J,W}). 

A finite group scheme $G$ is called {\em \'etale} if $G=G(k)$, where $G(k)$ is the group of closed points of $G$. Since $k$ is  algebraically closed, there is a categorical equivalence 
\[ \big\{ \text{\'Etale group schemes} \big\} \longleftrightarrow \big\{ \text{Finite abstract groups}\big\}\]
that assigns to an \'etale group scheme $G$ the group $G(k)$, and to a finite group $\Gamma$ the \'etale group scheme represented by $\O(\Gamma)$ (see, e.g., \cite[Section 6.4]{W} for more details).
 
A finite group scheme $G$ is {\em connected} if $G(k)=1$. Assume that $p>0$, and let $F\colon G\to G^{(1)}$ denote the Frobenius map. Since $G$ is finite and connected, it is annihilated by some Frobenius power $F^n \colon G\to G^{(n)}$, and a filtration by normal closed subgroup schemes 
\[1\triangleleft \Ker(F)\triangleleft \dots \triangleleft \Ker(F^n)=G,\]
where each composition factor $\Ker(F^k)/\Ker(F^{k-1})$ is a group scheme of \emph{height one}, i.e. with vanishing Frobenius.

\begin{example}\label{repliealg}
Consider a finite dimensional $p$-Lie algebra $\mathfrak g$ over $k$, and let $u(\mathfrak{g})$ denote its $p$-restricted universal enveloping algebra, which is a finite dimensional cocommutative Hopf algebra; see, e.g.,  \cite{SF}. Then the dual commutative Hopf algebra $u(\mathfrak{g})^*$ is local, and satisfies $x^p=0$ for all $x$ in the augmentation ideal.
\end{example}

By \cite[Section 2.7]{DG}, there is a categorical equivalence 
	\[ \big\{ \text{Height one finite group schemes} \big\} \longleftrightarrow \big\{ \text{Finite dim $p$-Lie algebras}\big\}\]
that assigns to a height one group scheme $G$ its Lie algebra $\Lie (G)$ with $p$-operation given by the differential of $x\mapsto x^p$. Conversely, the height one  group scheme corresponding to a finite dimensional $p$-Lie algebra $\mathfrak g$ is represented by $u(\mathfrak{g})^*$.

By a theorem of Cartier (see \cite[Section 11.4]{W}), if $p=0$ then every finite group scheme is \'etale. On the other hand, if $p>0$ then any finite group scheme $G$ fits into a split exact sequence
\begin{equation}\label{ses0}
1\to G^{\circ}\xrightarrow{i} G \mathrel{\mathop{\rightleftarrows}^{\pi}_{q}} G(k)\to 1,
\end{equation}
where $G^{\circ}$ is connected and $G(k)$ is \'etale, i.e., any finite group scheme is an extension of a connected group scheme by an \'etale one.
In particular, $G(k)$ acts on ${\rm Coh}(G)_1\cong{\rm Coh}(G^{\circ})$, $a\mapsto T_a$, and  
\begin{equation}\label{G(k)crossedproduct}
{\rm Coh}(G)\cong{\rm Coh}(G^{\circ})\rtimes G(k)
\end{equation}
is a crossed product category. Namely, ${\rm Coh}(G)={\rm Coh}(G^{\circ})\boxtimes {\rm Coh}(G(k))$ as abelian categories, and for any $X_1,X_2\in {\rm Coh}(G^{\circ})$ and $a_1,a_2\in G(k)$, we have
$$\left(X_1\boxtimes a_1 \right)\ot \left(X_2\boxtimes a_2 \right)= \left(X_1\ot T_{a_1}\left(X_2\right)\right)\boxtimes a_1a_2.$$

Thus, the classes of connected and \'etale group schemes play an important role in this paper. Restricting to these classes allows us, on the one hand, to illustrate our general results, and on the other hand to give a more specialized description of our objects of interest.

\subsection{Quotients and free actions}\label{sec:quotmorphind}
Let $H$ be a finite group scheme, and $X$ a finite scheme equipped with a {\em right} $H$-action 
\begin{equation}\label{mu}
\mu:=\mu_{X\times H}:X\times H\to X.
\end{equation}
Equivalently, the algebra homomorphism 
\begin{equation}\label{musharp}
\mu^{\sharp}:\O(X)\to \O(X)\ot \O(H)
\end{equation}
endows the algebra $\O(X)$ with a structure of a right $\O(H)$-comodule algebra (see, e.g., \cite{A}).
Since $H$ is finite, there exists a geometrical quotient finite scheme  
\begin{equation}\label{quotient morphismx}
\pi:X\twoheadrightarrow X/H,
\end{equation}
with coordinate algebra 
\begin{equation}\label{doublecosetfa}
\O(X/H)=\O(X)^{H}:=\{f\in\O(X)\mid \mu^{\sharp}(f)=f\ot 1\}.
\end{equation}

Recall that the action of $H$ on $X$ (\ref{mu}) is called {\em free} if the morphism 
\begin{equation*}
(p_1,\mu):X\times H\to X\times X
\end{equation*} 
is a closed immersion. Recall \cite[Theorem 1(B), p.112]{Mum} that in this case, the morphism $(p_1,\mu)$ induces a scheme isomorphism 
\begin{equation*}
X\times H\xrightarrow{\cong} X\times_{X/H} X.
\end{equation*}
Equivalently, $(p_1,\mu)$ induces an algebra isomorphism 
\begin{equation*}
\O(X)\otimes_{\O(X/H)} \O(X)\xrightarrow{\cong} \O(X)\otimes_k \O(H),
\end{equation*}
given by
$f\ot_{\O(X/H)}\tilde{f}\mapsto (f\ot 1)\mu^{\sharp}(\tilde{f});\,\,\,f,\tilde{f}\in\O(X)$.

\subsection{Equivariant sheaves}\label{sec:Equivariant sheaves}
Retain the notation of \S\ref{sec:quotmorphind}. Let ${\rm m}$ be the multiplication map of $H$, and set
\begin{equation}\label{eta}
\nu:=\mu(\id_X\times {\rm m})=\mu(\mu\times \id_{H}):
X\times H\times H\to X.
\end{equation}
Consider the obvious projections
\begin{gather*}
p_{1}:X\times H\twoheadrightarrow X,\,\,\,\text{p}_{1}:X\times H\times H\twoheadrightarrow X,\\
{\rm and}\,\,\,p_{12}:X\times H\times H\twoheadrightarrow X\times H. 
\end{gather*}
Clearly, $p_{1}\circ p_{12}=\text{p}_{1}$.

Let $\psi\in Z^2(H,\mathbb{G}_m)$ be a normalized $2$-cocycle, i.e., a Drinfeld twist for $\O(H)$. Namely, $\psi\in \mathscr{O}(H)^{\ot 2}$ is an
invertible element that satisfies the equations
$$(\Delta\ot \id)(\psi)(\psi\ot 1)=(\id\ot \Delta)(\psi)(1\ot \psi),$$
$$(\varepsilon\ot \id)(\psi)=(\id\ot \varepsilon)(\psi)=1.$$ Note that multiplication by $\psi$ defines an automorphism of any  sheaf on $H\times H$, which we still denote by $\psi$. Note also that $Z^2(H,\mathbb{G}_m)$ is a group since $\O(H)$ is commutative. 

\begin{definition}\label{defequiv}
Let $(H,\psi)$ and $X$ be as above.
\begin{enumerate}
\item
A $(H,\psi)$-equivariant sheaf on $X$ consists of a pair $(S,\rho)$, where $S\in {\rm Coh}(X)$ and $\rho$ is an
isomorphism $\rho:p_1^*S\xrightarrow{\cong} \mu^*S$ of sheaves on
$X\times H$ such that the diagram
\begin{equation*}
\xymatrix{p_1^*S \ar[d]_{(\id\times
{\rm m})^*(\rho)}\ar[rr]^{p_{12}^*(\rho)}
&& (\mu\circ p_{12})^*S\ar[d]^{(\mu\times \id)^*(\rho)} &&\\
\nu^*S\ar[rr]_{\id\boxtimes \psi} && \nu^*S&&}
\end{equation*}
of morphisms of sheaves on $X\times H\times H$ is commutative.
\item
Let $(S,\rho),\,(T,\tau)$ be two
$(H,\psi)$-equivariant  sheaves on $X$. A morphism $\phi:S\to T$ in
${\rm Coh}(X)$ is $(H,\psi)$-equivariant if the diagram
\begin{equation*}
\label{equivariantX2} \xymatrix{p_1^*S
\ar[d]_{\rho}\ar[rr]^{p_1^*(\phi)}
&& p_1^*T\ar[d]^{\tau} &&\\
\mu^*S\ar[rr]_{\mu^*(\phi)} && \mu^*T&&}
\end{equation*}
of morphisms of sheaves on $X\times H$ commutes.
\item
Let ${\rm Coh}^{(H,\psi)}(X)$ denote the category of
$(H,\psi)$-equivariant  sheaves on $X$, with
$(H,\psi)$-equivariant morphisms.
\end{enumerate}
\end{definition}

\begin{remark}
Morally, a $(H,\psi)$-equivariant  sheaf on $X$ is one equipped with a \emph{lift} of the $H$-action on $X$ to a $H$-action on the  sheaf, which is consistent with the automorphism of  sheaves on $H\times H$ induced by $\psi$.
\end{remark}

\begin{example}\label{ex0}
Consider $\O(X)\ot\O(H)$ as an $\O(X)$-module via 
the algebra map $\mu^{\sharp}$ (\ref{musharp}). Then $\mu^{\sharp}$ determines a $H$-equivariant structure on $\O(X)$ via the adjunction
\begin{eqnarray*}
\lefteqn{\Hom_{\O(X)}(\O(X),\O(X)\ot\O(H))}\\
& = & \Hom_{\O(X)}(\O(X),\mu_*p_1^*\O(X))=\Hom_{\O(X\times H)}(\mu^{*}\O(X),p_{1}^*\O(X)).
\end{eqnarray*}
However, for a nontrivial $\psi$, it is not always the case that $\O(X)$ admits a canonical $(H,\psi)$-equivariant structure (see Example \ref{ex1} below).
\end{example}

\begin{example}\label{ex1}
Let $H_\psi$ denote the finite group scheme central extension of $H$ by $\mathbb{G}_m$ associated to the $2$-cocycle $\psi$. Let $\Rep_k(H,\psi)$ denote the category of finite dimensional representations of $H_\psi$ on which $\mathbb{G}_m$ acts trivially. Then there is an equivalence of abelian categories 
$$\Rep_k(H,\psi)\cong {\rm Coh}^{(H,\psi)}({\rm pt}).$$ 
\end{example}

\subsection{Equivariant sheaves and comodules}\label{sec:Equivariant sheaves and comodules}
Given a finite group scheme $H$ and a twist $\psi\in \O(H)^{\ot 2}$, we let $\mathscr{O}(H)_{\psi}$ be the coalgebra with underlying vector space $\mathscr{O}(H)$ and comultiplication 
$\Delta_{\psi}$ given by
$\Delta_{\psi}(f):=\Delta(f)\psi$, where $\Delta$ is the standard comultiplication of $\mathscr{O}(H)$. It is clear that $\mathscr{O}(H)_{\psi}$ is a $H$-coalgebra, which is isomorphic to the regular representation of $H$ as an $H$-module. 

Note that the category $\Rep_k(H,\psi)$ (see Example \ref{ex1}) is equivalent to the category ${\rm Comod}(\mathscr{O}(H)_{\psi})_k$ of finite dimensional {\em right} $k$-comodules over $\mathscr{O}(H)_{\psi}$.

\begin{definition}\label{defright}
Let $X$ be as in \S\ref{sec:Equivariant sheaves}. Let 
${\rm Comod}(\mathscr{O}(H)_{\psi})_{{\rm Coh}(X)}$ be the abelian category of {\em right} $\mathscr{O}(H)_{\psi}$-comodules in ${\rm Coh}(X)$. Namely, the objects of this category are pairs $(S,\rho)$, where $S\in {\rm Coh}(X)$ and  
$\rho:S\to S\ot \mathscr{O}(H)_{\psi}$, such that 
$$\rho(f\cdot s)=\mu^{\sharp}(f)\cdot\rho(s);\,\,\,f\in\O(X),\,s\in S$$ 
(see (\ref{musharp})), and 
$(\rho\ot\id)\rho=(\id\ot\Delta_{\psi})\rho$. The morphisms in this category are ones that preserve the actions and coactions.
\qed
\end{definition}

\begin{proposition}\label{abelcat}
The following hold:
\begin{enumerate}
\item
There is a $k$-linear equivalence of categories 
$${\rm Coh}^{(H,\psi)}(X)\cong {\rm Comod}(\mathscr{O}(H)_{\psi})_{{\rm Coh}(X)}.$$
In particular, ${\rm Coh}^{(H,\psi)}(X)$ is an abelian category.
\item
If $\mathscr{I}\subset \O(X)$ is a $H$-stable ideal, then for any $S\in {\rm Coh}^{(H,\psi)}(X)$, $\mathscr{I}S\subset S$ is a subobject of $S$ in ${\rm Coh}^{(H,\psi)}(X)$.
\end{enumerate}
\end{proposition}

\begin{proof}
(1) The proof is similar to \cite[Proposition 3.7(3)]{G}.

(2) Follows from (1) since the equivariant structure of $S$ restricts to $\mathscr{I}S$ via the $H$-equivariant inclusion morphism $\mathscr{I}S\hookrightarrow S$.
\end{proof}

\subsection{Equivariant morphisms}\label{sec:Equivariant morphisms}
Assume $X,\,Y$ are finite schemes on which a finite group scheme $H$ acts from the right via $\mu_X:=\mu_{X\times H}$, $\mu_Y:=\mu_{Y\times H}$, respectively. Let $\psi\in Z^2(H,\mathbb{G}_m)$.

\begin{proposition}\label{equivmorphs}
For any $H$-equivariant morphism $\varphi:X\to Y$, the functors $\varphi^*$ (\ref{invimage}) and $\varphi_*$ (\ref{dirimage}) lift to a pair $(\varphi^*,\varphi_*)$ of adjoint functors of abelian categories 
$$\varphi^*:{\rm Coh}^{(H,\psi)}(Y)\to {\rm Coh}^{(H,\psi)}(X),\,\,\,\,\,\,\varphi_*:{\rm Coh}^{(H,\psi)}(X)\to {\rm Coh}^{(H,\psi)}(Y).$$ 
\end{proposition}

\begin{proof}
Let $(S,\rho)\in {\rm Coh}^{(H,\psi)}(Y)$, i.e., $S\in {\rm Coh}(Y)$ and $\rho$ is an $\O(Y)$-morphism endowing $S$ with a structure $\rho:S\to S\ot \O(H)_{\psi}$ of a right $\O(H)_{\psi}$-comodule. It is straightforward to verify that
$$\rho^*:=\mu_X^{\sharp}\bar{\ot} \rho:\O(X)\ot_{\O(Y)}S\to \O(X)\ot_{\O(Y)}S\ot \O(H)_{\psi}$$
endows $\varphi^*S$ with a structure of an object 
$(\varphi^*S,\rho^*)$ in ${\rm Coh}^{(H,\psi)}(X)$.

Let $(T,\tau)\in {\rm Coh}^{(H,\psi)}(X)$, i.e., $T\in {\rm Coh}(X)$ and $\tau$ is an $\O(X)$-morphism endowing $T$ with a structure $\tau:T\to T\ot \O(H)_{\psi}$ of a right $\O(H)_{\psi}$-comodule. It is easy to verify that $\tau_*:=\tau$ endows $\varphi_*T$ with a structure of an object
$(\varphi_*T,\tau_*)$ in ${\rm Coh}^{(H,\psi)}(Y)$.

It is straightforward to verify that the adjunction isomorphism
\begin{align*}
\Hom_{{\rm Coh}(X)}(\varphi^*S, T)\cong \Hom_{{\rm Coh}(Y)}(S, \varphi_*T) 
\end{align*}
preserves $(H,\psi)$-equivariant maps (see Definition \ref{defequiv}(2)), hence lifts to an isomorphism
\begin{align*}
\Hom_{{\rm Coh}^{(H,\psi)}(X)}((\varphi^*S,\rho^*), (T,\tau))\cong \Hom_{{\rm Coh}^{(H,\psi)}(Y)}((S,\rho), (\varphi_*T,\tau_*)), 
\end{align*}
as claimed.
\end{proof}

\begin{example}\label{fromreptocoh}
Let ${\rm u}:X\to {\rm Spec}(k)=\{{\rm pt}\}$ be the scheme structure morphism. Since $u$ is $H$-equivariant, we have a functor
$${\rm u}^*:\Rep_k(H,\psi)\to {\rm Coh}^{(H,\psi)}(X)$$
(see Example \ref{ex1}).
\end{example}

\subsection{Principal homogeneous spaces}\label{sec:Right principle homogeneous spaces}
In this section we assume that $\iota:H\xrightarrow{1:1}G$ is an embedding of finite group schemes. Consider the free right action of $H$ on $G$, given by
\begin{equation}\label{freerightactionhong}
\mu_{G\times H}:G\times H\to G,\,\,\,(g,h)\mapsto gh \footnote{Whenever there is no confusion, we will write $gh$ instead of $g\iota(h)$.},
\end{equation}
and (see \S\ref{sec:quotmorphind}) the corresponding quotient morphism 
\begin{equation}\label{quotient morphism}
\pi:G\twoheadrightarrow G/H.
\end{equation}

Recall that $\O(G/H)\subset \O(G)$ is a {\em right} $\O(H)$-Hopf Galois cleft extension. Namely, $\O(G)$ is a right $\O(H)$-comodule algebra via $\mu_{G\times H}^{\sharp}$, $\O(G/H)\subset \O(G)$ is the {\em left} coideal subalgebra of coinvariants, the map 
\begin{equation*}
\O(G)\ot_{\O(G/H)}\O(G)\to \O(G)\ot_k\O(H),\,\,\,{\rm f}\ot \tilde{{\rm f}}\mapsto ({\rm f}\ot 1)\mu_{G\times H}^{\sharp}(\tilde{{\rm f}}),
\end{equation*}
is bijective, and there exists a (unique up to multiplication by a convolution invertible element in $\Hom_k(\O(H),\O(H\backslash G))$) unitary convolution invertible right $\O(H)$-colinear map, called the {\em cleaving map},  
\begin{equation}\label{nbpgamma0}
\mathfrak{c}:\O(H)\xrightarrow{1:1} \O(G).
\end{equation}
A {\em choice} of $\mathfrak{c}$ determines a convolution invertible $2$-cocycle 
\begin{equation}\label{sigma}
\sigma:\O(H)^{\ot\,2}\to \O(G/H),\,\,\,f\ot \tilde{f}\mapsto \mathfrak{c}(f_1)\mathfrak{c}(\tilde{f}_1)\mathfrak{c}^{-1}(f_1\tilde{f}_2),
\end{equation}
from which one can define the crossed product algebra $\O(G/H)\#_{\sigma}\O(H)$ (with the trivial action of $\O(H)$ on $\O(G/H)$), which is naturally a left $\O(G/H)$-module and right $\O(H)$-comodule algebra, where $\O(H)$ coacts via $a\#f\mapsto a\#f_1\ot f_2$. 
Moreover, the map 
\begin{equation}\label{newestdeofphi}
\phi:\O(G/H)\#_{\sigma}\O(H)\xrightarrow{\cong} \O(G),\,\,\,a\#f\mapsto a\mathfrak{c}(f),
\end{equation}
is an algebra isomorphism, $\O(G/H)$-linear and right $\O(H)$-colinear. The inverse of $\phi$ is given by
\begin{equation}\label{newdeofphi-1}
\phi^{-1}:\O(G)\xrightarrow{\cong} \O(G/H)\#_{\sigma}\O(H),\,\,\,{\rm f}\mapsto {\rm f}_1\mathfrak{c}^{-1}(\iota^{\sharp}({\rm f}_2))\#\iota^{\sharp}({\rm f}_3).
\end{equation}
We also have the $\O(G/H)$-linear map
\begin{equation}\label{nbpalphatilderp}
\alpha:=(\id\bar{\ot}\varepsilon)\phi^{-1}:\O(G)\twoheadrightarrow \O(G/H),\,\,\,{\rm f}\mapsto {\rm f}_1 \mathfrak{c}^{-1}(\iota^{\sharp}({\rm f}_2)).
\end{equation}
(For details, see \cite[Section 3]{Mon} and references therein.)

Since $\sigma$ (\ref{sigma}) is a $2$-cocycle, it follows that for any $S\in {\rm Coh}(G/H)$, 
the vector space $S\ot_k\O(H)$ is an $\O(G/H)\#_{\sigma}\O(H)$-module via
$$(a\#f)\cdot \left(s\ot \tilde{f}\right)=a\sigma(f_1,\tilde{f}_1)\cdot s\ot f_2\tilde{f}_2.$$
Consequently, we have the following lemma.

\begin{lemma}\label{helpful2-2}
For any $U\in {\rm Coh}(G/H)$, the following hold:
\begin{enumerate}
\item
The vector space $U\ot_k\O(H)$ is an $\O(G)$-module via $\phi^{-1}$ (\ref{newdeofphi-1}): 
$${\rm f}\cdot \left(u\ot \tilde{f}\right):=\phi^{-1}\left({\rm f}\right)\cdot \left(u\ot \tilde{f}\right);\,\,{\rm f}\in \O(G),\, u\ot\tilde{f}\in U\ot \O(H).$$
\item
We have an $\O(G)$-linear isomorphism 
$$F_U:U\ot_{\O(G/H)}\O(G)\xrightarrow{\cong} U\ot_k\O(H),\,\,\,u\ot_{\O\left(G/H\right)}{\rm f}\mapsto \phi^{-1}\left({\rm f}\right)\cdot \left(u\ot 1\right),$$
whose inverse is given by
$$F_U^{-1}:U\ot_k\O(H)\xrightarrow{\cong} U\ot_{\O(G/H)}\O(G),\,\,\,u\ot_k f\mapsto u\ot_{\O(G/H)} \phi(1\ot_k f).$$
\end{enumerate} 
\end{lemma}

\begin{proof}
Follow from the preceding remarks.
\end{proof}

For any $U\in {\rm Coh}(G/H)$, define 
\begin{equation}\label{yetanotherrhoU}
\rho^{\psi}_U:=(F_U^{-1}\ot\id)(\id\ot\Delta_{\psi})F_U.
\end{equation}

For any $(S,\rho)\in {\rm Coh}^{(H,\psi)}(G)$, define the subspace of coinvariants 
\begin{equation}\label{firstcoinvariants}
S^{(H,\psi)}:=\{s\in S\mid \rho(s)=\mathfrak{c}(\psi^1)\cdot s\ot \psi^2\}\subset S.
\end{equation}

\begin{theorem}\label{helpful2}
The following hold:
\begin{enumerate}
\item 
There is an equivalence of abelian categories 
$${\rm Coh}^{(H,1)}(G)\xrightarrow{\cong}{\rm Coh}^{(H,\psi)}(G),\,\,(S,\rho)\mapsto \left(S^{(H,1)}\ot_k \O(H),\id\ot\Delta_{\psi}\right)$$
(see \ref{yetanotherrhoU}), whose inverse is given by
$${\rm Coh}^{(H,\psi)}(G)\xrightarrow{\cong}{\rm Coh}^{(H,1)}(G),\,\,
(S,\rho)\mapsto \left(S^{(H,\psi)}\ot_k \O(H),\id\ot\Delta\right).$$ 
\item
The quotient morphism $\pi:G\twoheadrightarrow G/H$ (\ref{quotient morphism}) induces an equivalence of abelian categories
$$\pi^*:{\rm Coh}(G/H)\xrightarrow{\cong}{\rm Coh}^{(H,\psi)}(G),\,\,\,
U\mapsto \left(U\ot_{\O(G/H)}\O(G),\rho^{\psi}_U\right),$$
whose inverse is given by
$$\pi_*^{(H,\psi)}:{\rm Coh}^{(H,\psi)}(G)\xrightarrow{\cong}{\rm Coh}(G/H),\,\,\,
(S,\rho)\mapsto \pi^{(H,\psi)}_*S.$$
\end{enumerate}
\end{theorem}

\begin{proof}
(1) Follows from Lemma \ref{helpful2-2} in a straightforward manner.  

(2) For $\psi=1$, this is \cite[p.112]{Mum}, so the claim follows from (1).
\end{proof}

Now for any closed point $\bar{g}:=gH\in (G/H)(k)$, let $\delta_{\bar{g}}$ denote the corresponding simple object of ${\rm Coh}(G/H)$ (see \S\ref{sec:coh decomposition}). Let  
\begin{equation}\label{varioussimplespre}
S_{\bar{g}}:=\pi^*\delta_{\bar{g}}\cong \delta_{\bar{g}}\ot_k \O(H)\in {\rm Coh}^{(H,\psi)}(G)
\end{equation}
(see Theorem \ref{helpful2}), and let 
\begin{equation}\label{variouspcsimplespre}
P_{\bar{g}}:=P(\delta_{\bar{g}})\in {\rm Coh}(G/H),\,\,\,{\rm and}\,\,\,P(S_{\bar{g}})\in {\rm Coh}^{(H,\psi)}(G)
\end{equation}
be the projective covers of $\delta_{\bar{g}}$ and $S_{\bar{g}}$, respectively.

\begin{corollary}\label{simprojhpsi}
The following hold:
\begin{enumerate}
\item
For any $\bar{g}\in (G/H)(k)$, there are ${\rm Coh}^{(H,\psi)}(G)$-isomorphisms 
\begin{equation*}
S_{\bar{g}}\cong \delta_{g}\ot S_{\bar{1}},\,\,\,{\rm and}\,\,\,P_{\bar{g}}\cong \delta_{g}\ot P_{\bar{1}},
\end{equation*}
where $\O(G)$ acts diagonally and $\O(H)_{\psi}$ coacts on the right.
\item
The assignment $\bar{g}\mapsto S_{\bar{g}}$ is a bijection between $(G/H)(k)$ and the set of isomorphism classes of simples in ${\rm Coh}^{(H,\psi)}(G)$.
\item
We have $P(S_{\bar{g}})\cong \pi^*P_{\bar{g}}\cong P_{gH}:=\bigoplus_{h\in H(k)}P_{gh}$.
\end{enumerate}
\end{corollary}

\begin{proof}
Follow immediately from Lemma \ref{helpful2-2} and Theorem \ref{helpful2}.
\end{proof}

\begin{lemma}\label{ogisohbicom0}
Let $\Psi\in \O(G)^{\ot 2}$ be a twist, and set   
$\psi:=\left(\iota^{\sharp}\ot \iota^{\sharp}\right)(\Psi)$.  
Then $\psi\in \O(H)^{\ot 2}$ is a twist, and the linear map 
$$\rho_{\Psi}:\O(G)\to \O(G)\ot \O(H)_{\psi},\,\,\,f\mapsto  f_1\Psi^1\ot \iota^{\sharp}\left(f_2\Psi^2\right),$$ 
endows $\O(G)$ with a structure of a right $\O(H)_{\psi}$-comodule in ${\rm Coh}(G)$, with respect to the right action $\mu_{G\times H}$ (\ref{freerightactionhong}).
\end{lemma}

\begin{proof}
Follows since $\iota^{\sharp}:\O(G)_{\Psi}\twoheadrightarrow \O(H)_{\psi}$ is a coalgebra map.
\end{proof}

Consider next the free {\em left} action of $H$ on $G$ given by 
\begin{equation}\label{freerightactionhong2}
\widetilde{\mu_{H\times G}}:H\times G\to G,\;\;\;(h,g)\mapsto hg,
\end{equation}
and the corresponding quotient morphism  
\begin{equation}\label{quotient morphism2}
\widetilde{\pi}:G\twoheadrightarrow H\backslash G.
\end{equation}
Similarly to the above, we can {\em choose} a (unique up to multiplication by an invertible element in $\Hom_k(\O(H),\O(H\backslash G))$) cleaving map
\begin{equation}\label{nbpgammatilde}
\widetilde{\mathfrak{c}}:\O(H)\xrightarrow{1:1} \O(G),
\end{equation}
which then determines a $2$-cocycle $\widetilde{\sigma}$ and an $\O(H\backslash G)$-linear and left $\O(H)$-colinear isomorphism of algebras
\begin{equation}\label{nbpphitilde}
\widetilde{\phi}:\O(H\backslash G)\#_{\widetilde{\sigma}}\O(H)\xrightarrow{\cong} \O(G),\,\,\,a\# f\mapsto a \widetilde{\mathfrak{c}}(f),
\end{equation}
whose inverse is given by
\begin{equation}\label{nbpphitildeinv}
\widetilde{\phi}^{-1}:\O(G)\xrightarrow{\cong} \O(H\backslash G)\#_{\widetilde{\sigma}}\O(H),\,\,\,{\rm f}\mapsto {\rm f}_1 \widetilde{\mathfrak{c}}^{-1}(\iota^{\sharp}({\rm f}_2))\#\iota^{\sharp}({\rm f}_3),
\end{equation}
and an $\O(H\backslash G)$-linear map
\begin{equation}\label{nbpalphatilde0}
\widetilde{\alpha}:=(\id\bar{\ot}\varepsilon)\widetilde{\phi}^{-1}:\O(G)\twoheadrightarrow \O(H\backslash G),\,\,\,{\rm f}\mapsto {\rm f}_1 \widetilde{\mathfrak{c}}^{-1}(\iota^{\sharp}({\rm f}_2)).
\end{equation}
%

\subsection{Twisting by the inverse $2$-cocycle}\label{sec:Twisting by the inverse $2$-cocycle} 
Retain the notation of \S\ref{sec:Right principle homogeneous spaces}.  
We note an explicit relation between the twisted coalgebras associated to $\psi$ and $\psi^{-1}$.
  
Set $Q_{\psi}:=\sum {\rm S}(\psi^{1})\psi^{2}$, where ${\rm S}$ is the antipode map of $\O(H)$. It is well known (see, e.g., \cite{AEGN,Ma}) that  
\begin{equation*}\label{sspsi21}
Q_{\psi}^{-1}=\sum \psi^{-1}{\rm S}(\psi^{-2})\,\,\,{\rm and}\,\,\,\Delta(Q_{\psi})=({\rm S}\ot {\rm S})(\psi_{21}^{-1})(Q_{\psi}\ot Q_{\psi})\psi^{-1}.
\end{equation*}
In particular, 
$\Delta(Q_{\psi}^{-1}{\rm S}(Q_{\psi}))=Q_{\psi}^{-1}{\rm S}(Q_{\psi})\ot Q_{\psi}^{-1}{\rm S}(Q_{\psi})$,
i.e., $Q_{\psi}^{-1}{\rm S}(Q_{\psi})$ is a grouplike element of $\O(H)$. Thus,   
we have a coalgebra isomorphism 
$$
\mathscr{O}(H)_{\psi}^{{\rm cop}}\xrightarrow{\cong} \mathscr{O}(H)_{\psi^{-1}},\,\,\,f\mapsto {\rm S}(f Q_{\psi}^{-1}),
$$
which induces the following equivalence of abelian categories 
\begin{equation}\label{fromlefttoright}
{\rm Corep}(H,\psi^{-1})_k\xrightarrow{\cong}{\rm Corep}_k(H,\psi),\,\,\,(V,r)\mapsto (V,\widetilde{r}),
\end{equation}
where $\widetilde{r}:V\to \mathscr{O}(H)_{\psi}\ot V$, $v\mapsto {\rm S}(v)Q_{\psi}\ot v^0$.

\begin{lemma}\label{ogisohbicom}
Let $\Psi\in \O(G)^{\ot 2}$ be a twist, and set   
$\psi:=\left(\iota^{\sharp}\ot \iota^{\sharp}\right)(\Psi)$.  
Then $\psi\in \O(H)^{\ot 2}$ is a twist, and the linear map   
\begin{equation}\label{tilderhoPsi}
\widetilde{\rho_{\Psi}}:\O(G)\to \O(G)\ot \O(H)_{\psi},\,\,\,f\mapsto \Psi^{-2}f_2\ot \iota^{\sharp}{\rm S}\left(f_1\Psi^{-1}\right)Q_{\psi},
\end{equation} 
endows $\O(G)$ with a structure of a right $\O(H)_{\psi}$-comodule in ${\rm Coh}(G)$, with respect to the right action $G\times H\to G$, $(g,h)\mapsto h^{-1}g$.
\end{lemma}

\begin{proof}
Follows from the above and the fact that $\iota^{\sharp}:\O(G)_{\Psi}\twoheadrightarrow \O(H)_{\psi}$ is a coalgebra map. Indeed, set $\rho:=\widetilde{\rho_{\Psi}}$ and $Q:=Q_{\psi}$. Then on the one hand, we have
\begin{eqnarray*}
\lefteqn{(\rho\ot\id)\rho(f)=\rho(\Psi^{-2}f_2)\ot \iota^{\sharp}{\rm S}\left(f_1\Psi^{-1}\right)Q}\\
& = & \Omega^{-2}(\Psi^{-2}f_2)_2\ot \iota^{\sharp}{\rm S}\left((\Psi^{-2}f_2)_1\Omega^{-1}\right)Q\ot \iota^{\sharp}{\rm S}\left(f_1\Psi^{-1}\right)Q\\
& = & \Omega^{-2}\Psi^{-2}_2f_3\ot \iota^{\sharp}{\rm S}\left(\Psi^{-2}_1f_2\Omega^{-1}\right)Q\ot \iota^{\sharp}{\rm S}\left(f_1\Psi^{-1}\right)Q\\
& = & \left(\Omega^{-2}\Psi^{-2}_2\ot \iota^{\sharp}{\rm S}\left(\Psi^{-2}_1\Omega^{-1}\right)Q\ot \iota^{\sharp}{\rm S}\left(\Psi^{-1}\right)Q\right)
\widetilde{\mu}(f),
\end{eqnarray*}
and on the other hand, we have
\begin{eqnarray*}
\lefteqn{(\id\ot\Delta_{\psi})\rho(f)=\Psi^{-2}f_2\ot \Delta_{\psi}\left(\iota^{\sharp}{\rm S}\left(f_1\Psi^{-1}\right)Q)\right)}\\
& = & \Psi^{-2}f_2\ot \psi\left(\left(\iota^{\sharp}{\rm S}\left(f_1\Psi^{-1}\right)_1\ot \iota^{\sharp}{\rm S}\left(f_1\Psi^{-1}\right)_2\right)\Delta(Q)\right)\\
& = & \Psi^{-2}f_3\ot \psi\left(\left(\iota^{\sharp}{\rm S}\left(f_2\Psi^{-1}_2\right)\ot \iota^{\sharp}{\rm S}\left(f_1\Psi^{-1}_1\right)\right)\Delta(Q)\right)\\
& = & 
\left(\Psi^{-2}\ot \left(\iota^{\sharp}{\rm S}\left(\Psi^{-1}_2\right)\ot \iota^{\sharp}{\rm S}\left(\Psi^{-1}_1\right)\right)\Delta(Q)\psi\right)
\widetilde{\mu}(f).
\end{eqnarray*}

Thus, we have
\begin{eqnarray*}
\lefteqn{\Omega^{-2}\Psi^{-2}_2\ot \iota^{\sharp}{\rm S}\left(\Psi^{-2}_1\Omega^{-1}\right)Q\ot \iota^{\sharp}{\rm S}\left(\Psi^{-1}\right)Q}\\
& = & \Psi^{-2}\ot \iota^{\sharp}{\rm S}\left(\Psi^{-1}_2\Omega^{-2}\right)Q\ot \iota^{\sharp}{\rm S}\left(\Psi^{-1}_1\Omega^{-1}\right)Q\\
& = & 
\Psi^{-2}\ot \left(\iota^{\sharp}{\rm S}\left(\Psi^{-1}_2 \right)\ot \iota^{\sharp}{\rm S}\left(\Psi^{-1}_1\right)\right)
\left({\rm S}\ot {\rm S} \right)(\psi_{21}^{-1})
(Q\ot Q)\\
& = & 
\Psi^{-2}\ot \left(\iota^{\sharp}{\rm S}\left(\Psi^{-1}_2 \right)\ot \iota^{\sharp}{\rm S}\left(\Psi^{-1}_1\right)\right)\Delta(Q)\psi,
\end{eqnarray*}
which implies the claim.
\end{proof}

\subsection{Module categories over ${\rm Coh}(G)$} \label{sec:M(H,psi)}
Fix a finite group scheme $G$. Let $H\subset G$ be a closed subgroup scheme, and let $\psi\in Z^2(H,\mathbb{G}_m)$ be a normalized $2$-cocycle. Recall \cite{G} the abelian category  
$$\M(H,\psi):={\rm Coh}^{(H,\psi)}(G)$$
with respect to the right action (\ref{freerightactionhong}) of $H$ on $G$. 
Recall \cite{G} that $\M:=\M(H,\psi)$ admits a canonical structure of an indecomposable exact left module category over ${\rm Coh}(G)$ given by convolution of sheaves. Namely, the action of $X\in {\rm Coh}(G)$ on $(S,\rho)\in \M$ is given by 
$$X\ot^{\M} (S,\rho)=\left(X\ot S,\id\ot\rho\right).$$

\begin{example}\label{simpexs0}
The following hold:
\begin{enumerate}
\item
$\mathscr{M}(1,1)\cong {\rm Coh}(G)$ is the regular ${\rm Coh}(G)$-module.
\item
$\mathscr{M}(G,1)\cong \Vect$ is the usual fiber functor on ${\rm Coh}(G)$.
\end{enumerate}
\end{example}

For any closed point $g\in G(k)$, set 
\begin{equation}\label{psig}
\psi^g:={\rm Ad}g(\psi)\in Z^2(g^{-1}Hg,\mathbb{G}_m).
\end{equation}

\begin{definition}\label{defequivpairs} 
Two pairs $(H,\psi)$, $(K,\eta)$ as above are called equivalent if there is a closed point $g\in G(k)$ such that $K=g^{-1}Hg$ and $[\eta]=[\psi^g]$ in $H^2(K,\mathbb{G}_m)$.
\end{definition}

\begin{theorem}\cite{G}\label{modomega}
The following hold:
\begin{enumerate}
\item
There is a bijection between equivalence classes of pairs $(H,\psi)$ (in the sense of Definition \ref{defequivpairs}) and equivalence classes of indecomposable exact left module categories over
${\rm Coh}(G)$, assigning $(H,\psi)$ to $\mathscr{M}(H,\psi)$.
\item
There exists an equivalence of ${\rm Coh}(G)$-module categories 
$$\mathscr{M}(H,\psi)\cong {\rm Comod}(\mathscr{O}(H)_{\psi})_{{\rm Coh}(G)}.$$\end{enumerate}
\end{theorem}

\begin{remark}\label{helpful} 
If $G$ is connected then Theorem \ref{modomega} says that equivalence classes of indecomposable exact module categories over
${\rm Coh}(G)$ correspond bijectively to pairs $(H,\psi)$. \qed
\end{remark}

\begin{remark}\label{alsohelpful2}
Let $\M:=\M(H,\psi)$. Corollary \ref{simprojhpsi}(1) states that
$$S_{\bar{g}}\cong \delta_{g}\ot^{\M} S_{\bar{1}},\,\,\,{\rm and}\,\,\,P_{\bar{g}}\cong \delta_{g}\ot^{\M} P_{\bar{1}}$$ for any simple $S_{\bar{g}}\in \M$. Also, using Corollary \ref{simprojhpsi}(3), we see that 
\begin{equation*}
\begin{split}
& P_{1}\ot^{\M} S_{\bar{1}}\cong P_{1}\ot^{\M} \O(H)\cong \pi_*^{(H,\psi)}\pi^*\left(P_{1}\ot \O(H)\right)\\
& \cong \pi_*^{(H,\psi)}P_{H}\cong |H^{\circ}|P_{\bar{1}},
\end{split}
\end{equation*}
which in particular demonstrates that $\M$ is exact (see \cite{G}). \qed
\end{remark}

\begin{remark}\label{rem:N(H,psi)}
For any pair $(H,\psi)$ as above, one can consider also the abelian category  
\[\mathscr{N}(H, \psi):={\rm Comod}_{{\rm Coh}(G)}(\mathscr{O}(H)_{\psi})\]
of {\bf left} $\O(H)_{\psi}$-comodules in ${\rm Coh}(G)$ with respect to the left $H$-action $\mu_{H\times G}:H\times G\to G$ by left translation.  
Then similarly, the assignment $(H,\psi)\mapsto \mathscr{N}(H, \psi)$ classifies exact indecomposable {\bf right} module categories over ${\rm Coh}(G)$, up to equivalence.
\end{remark}
 
Recall \cite{G} that the group scheme-theoretical category 
$$\C(G,H,\psi):={\rm Coh}(G)^*_{\mathscr{M}(H,\psi)}$$
associated to this data is the dual category of ${\rm Coh}(G)$ with respect to $\mathscr{M}(H,\psi)$. In other words, 
\begin{equation}\label{gsctcdef}
\C(G,H,\psi)=\Fun_{{\rm Coh}(G)}\left(\mathscr{M}(H,\psi),\mathscr{M}(H,\psi)\right)
\end{equation}
is the category of ${\rm Coh}(G)$-module endofunctors of $\mathscr{M}(H,\psi)$.
Recall that $\C(G,H,\psi)$ is a finite tensor category with tensor product given by composition of module functors \cite{EO}.

Recall \cite{G} that convolution of sheaves on $G$ lifts to endow the category $\M\left(H\times H,\psi^{-1}\times \psi\right)$ with the structure of a finite tensor category. 

Finally, let ${\rm Bicomod}_{{\rm Coh}(G)}(\mathscr{O}(H)_{\psi})$ be the finite tensor category with tensor product $\ot^{\mathscr{O}(H)_{\psi}}$ and unit object $\mathscr{O}(H)_{\psi}$. Recall that objects in this category are triples $(V,\lambda,\rho)$ where $(V,\lambda)$ is a left $\O(H)_{\psi}$-comodule in ${\rm Coh}(G)$ with respect to the left action $\mu_{H\times G}$, and $(V,\rho)$ is a right $\O(H)_{\psi}$-comodule in ${\rm Coh}(G)$ with respect to the right action $\mu_{G\times H}$.

\begin{lemma}\label{equdefofgstc}
There are equivalences of tensor categories
\begin{equation*}
\C(G,H,\psi)\cong{\rm Bicomod}_{{\rm Coh}(G)}(\mathscr{O}(H)_{\psi})^{{\rm rev}},
\end{equation*}
and
$$\C(G,H,\psi)\cong \M(H\times H,\psi^{-1}\times \psi).
$$
\end{lemma}

\begin{proof}
All equivalences are standard. For the first one see, e.g., \cite{EGNO}, and for the second one see, e.g., Lemma \ref{newsimple} below.
\end{proof}

\begin{example}\label{simpexs}
The following hold: 
\begin{enumerate} 
\item
$\C\left(G,1,1\right)\cong {\rm Coh}(G)$.
\item
$\C\left(G,G,1\right)\cong {\rm Rep}(G)$.
\end{enumerate}
\end{example}

\section{Equivariant sheaves on group schemes}\label{S:Biequivariant sheaves on group schemes}

Let $\partial:A\xrightarrow{1:1} B$ be an embedding of finite group schemes. Let $\Psi\in Z^2(B,\mathbb{G}_m)$, and $\xi:=\Psi_{\mid A}$ be its restriction to $A$ via $\partial$, i.e., 
\begin{equation}\label{Psirestopsi}
\xi:=\left(\partial^{\sharp}\ot \partial^{\sharp}\right)\left(\Psi\right)\in Z^2(A,\mathbb{G}_m).
\end{equation}
  
Let ${\rm Coh}^{(A\times B,\xi^{-1}\times \Psi)}(B)$ be the category of $(A\times B,\xi^{-1}\times \Psi)$-equivariant sheaves on $B$ with respect to the right action
$$\mu_{B\times (A\times B)}:B\times (A\times B)\to B,\,\,\,(b_1,a,b_2)\mapsto \partial(a)^{-1}b_1b_2,$$ and ${\rm Coh}^{(B,\Psi)}(A\backslash B)$ the category of $(B,\Psi)$-equivariant sheaves on $A\backslash B$ with respect to the right action
$$\mu_{A\backslash B\times B}:A\backslash B\times B\to A\backslash B,\,\,\,(\overline{b_1},b_2)\mapsto \overline{b_1b_2},$$ where $A\backslash B$ is the quotient scheme with respect to $\widetilde{\mu_{A\times B}}$ \eqref{freerightactionhong2}.
Our goal is to construct explicit equivalences of abelian categories
\begin{equation*}
	\begin{tikzcd}[column sep=4em]
		\mathrm{Rep}_k(A,\xi^{-1}) 
		\arrow[r, "{\rm t}^*", shift left=0.35ex] 
		\arrow[rr, "{\rm Ind}_{(A,\xi^{-1})}^{(B,\Psi)}", bend left=16, shift left=0.5ex]
		&
		\mathrm{Coh}^{(A\times B,\xi^{-1}\times \Psi)}(B) 
		\arrow[l, "{\rm t}_*^{(B,\Psi)}", shift left=0.35ex] 
		\arrow[r, "{\rm p}_*^{(A,\xi^{-1})}", shift left=0.35ex]
		&
		\mathrm{Coh}^{(B,\Psi)}(A\backslash B) 
		\arrow[l, "{\rm p}^*", shift left=0.3ex]
		\arrow[ll, "{\rm Res}_{(A,\xi^{-1})}^{(B,\Psi)}", bend left=16, shift left=0.5ex]
	\end{tikzcd}
\end{equation*}

\subsection{The equivalence $\mathrm{Rep}_k(A,\xi^{-1}) \cong	\mathrm{Coh}^{(A\times B,\xi^{-1}\times \Psi)}(B)$}
For any $(S,\rho)\in{\rm Coh}^{(A\times B,\xi^{-1}\times \Psi)}(B)$, define the maps
\begin{equation}\label{defnoflambda1}
\begin{split}
& \rho_2:=(\id\ot\varepsilon\ot\id)\rho:S\to S\ot \O(B)_{\Psi},\\
& \rho_3:=(\id\ot\id\ot\varepsilon)\rho:S\to S\ot \O(A)_{\xi^{-1}},
\end{split}
\end{equation}
and the coinvariants subsheaf
\begin{equation}\label{coinvBPsi}
S^{(B,\Psi)}:=\{s\in S\mid \rho_2(s)=\Psi^1\cdot s\ot \Psi^2\}\subset S.
\end{equation}

\begin{lemma}\label{inducedinvariants} 
For any $(S,\rho)\in {\rm Coh}^{(A\times B,\xi^{-1}\times \Psi)}(B)$, the following hold:
\begin{enumerate}
\item 
$(S,\rho_2)\in {\rm Coh}^{(B,\Psi)}(B)$.
\item
${\rm Coh}^{(B,\Psi)}(B)\cong\Vect$, with the unique simple object being $\left(\O(B),\Delta_{\Psi}\right)$.
\item
There is a ${\rm Coh}^{(B,\Psi)}(B)$-isomorphism 
$$(S,\rho_2)\cong \left(\O(B),\Delta_{\Psi}\right)\otimes_k S^{(B,\Psi)}.$$
\item
$(S,\rho_3)\in {\rm Rep}_k (A,\xi^{-1})$, and $(S^{(B,\Psi)},\rho_3)\subset (S,\rho_3)$ in ${\rm Rep}_k (A,\xi^{-1})$.
\end{enumerate}
\end{lemma}

\begin{proof}
(1) is clear, (2) follows from \cite{G}, and (3)-(4) from (1)-(2). 
\end{proof}

Consider the trivial morphism ${\rm t}:B\twoheadrightarrow 1$.

\begin{theorem}\label{lemmafromg}\footnote{The assumption $\xi=\Psi_{\mid A}$ (\ref{Psirestopsi}) is not needed here.}
The functor ${\rm t}^*:\Vect \to {\rm Coh}(B)$ lifts to an equivalence of abelian categories   
$${\rm t}^*:\Rep_k(A,\xi^{-1})\xrightarrow{\cong}{\rm Coh}^{(A\times B,\xi^{-1}\times \Psi)}(B),\,\,\,V\mapsto \left(\O(B)\ot_k V,\rho^V\right)$$
(see (\ref{natural structure}) below), 
whose inverse is given by
$${\rm t}_*^{(B,\Psi)}:{\rm Coh}^{(A\times B,\xi^{-1}\times \Psi)}(B)\xrightarrow{\cong}\Rep_k(A,\xi^{-1}),\,\,\,(S,\rho)\mapsto \left(S^{(B,\Psi)},\rho_3\right).$$
\end{theorem}

\begin{proof}
Let $(V,r)\in \Rep_k(A,\xi^{-1})={\rm Coh}^{(A,\xi^{-1})}(1)$ be any object, and write $r(v)=\sum v^0\ot v^1\in V\ot \O(A)_{\xi^{-1}}$. Consider the free $\O(B)$-module ${\rm t}^*V=\O(B)\ot_k V$, and map
$$\rho^V:=\left((\widetilde{\rho_1}\ot\id)\Delta_{\Psi}\right)\bar{\ot}r$$
(see (\ref{tilderhoPsi})). 
Namely, we have
\begin{equation}\label{natural structure}
\begin{split}
& \rho^V:\O(B)\ot_k V\to \O(B)\ot_k V\ot \O(A)_{\xi^{-1}}\ot \O(B)_{\Psi},\\
& f\ot v\mapsto \Psi^1_2 f_2\ot v^{0}\ot \partial^{\sharp}{\rm S}\left(\Psi^1_1f_1\right)v^{1}\ot \Psi^2 f_3.
\end{split}
\end{equation} 
It is straightforward to verify that 
$$\left(\O(B)\ot_k V,\rho^V\right)\in {\rm Coh}^{(A\times B,\xi^{-1}\times \Psi)}(B).$$
Thus, we have defined a functor 
\begin{equation*}
\Rep_k(A,\xi^{-1})\to {\rm Coh}^{(A\times B,\xi^{-1}\times \Psi)}(B),\,\,\,V\mapsto \left(\O(B)\ot_k V,\rho^V\right).
\end{equation*}

Conversely, take any $(S,\rho)\in {\rm Coh}^{(A\times B,\xi^{-1}\times \Psi)}(B)$, and consider the sheaf ${\rm t}_{*}S$ (that is, the underlying vector space of $S$). Then by Lemma \ref{inducedinvariants}(4), $\left(S^{(B,\Psi)},\rho_3\right)\in\Rep_k(A,\xi^{-1})$. Thus, we have a functor 
\begin{equation*}
{\rm Coh}^{(A\times B,\xi^{-1}\times \Psi)}(B)\to \Rep_k(A,\xi^{-1}),\,\,\,(S,\rho)\mapsto \left(S^{(B,\Psi)},\rho_3\right).
\end{equation*}

Finally, it is straightforward to verify that the two functors constructed above are inverse to each other. 
\end{proof}

\begin{remark}\label{affinegrschs1}
If $B$ is any affine group scheme, then Theorem \ref{lemmafromg} and its proof hold after replacing ${\rm Coh}$ by ${\rm Coh}_{{\rm f}}$ (see \cite{G}).
\end{remark}

\subsection{The equivalence $\mathrm{Coh}^{(A\times B,\xi^{-1}\times \Psi)}(B) \cong {\rm Coh}^{(B,\Psi)}(A\backslash B)$}
Consider the free left action $A\times B\to B$, $(a,b)\mapsto \partial(a)b$, and let  
\begin{equation}\label{corrquotp}
{\rm p}:B\twoheadrightarrow A\backslash B
\end{equation}
be the corresponding quotient morphism.

For any $(S,\rho)\in{\rm Coh}^{(A\times B,\xi^{-1}\times \Psi)}(B)$, define the coinvariants subsheaf
\begin{equation}\label{coinvAxi}
S^{\left(A,\xi^{-1}\right)}:=\{s\in S\mid \rho_3(s)=\Psi^{2}\cdot s\ot \partial^{\sharp}{\rm S}(\Psi^{1})Q_{\xi^{-1}}\}\subset S
\end{equation}
(see (\ref{defnoflambda1})).

\begin{lemma}\label{inducedinvariantsquotient} 
For any $(S,\rho)\in {\rm Coh}^{(A\times B,\xi^{-1}\times \Psi)}(B)$, $(S,\rho_3)\in {\rm Coh}^{(A,\xi^{-1})}(B)$.
\end{lemma}

\begin{proof}
Similar to the proof of Lemma \ref{inducedinvariants}.
\end{proof}

\begin{theorem}\label{lemmatezer}
The functor ${\rm p}^*:{\rm Coh}(A\backslash B)\xrightarrow{}{\rm Coh}(B)$ (\ref{corrquotp}) lifts to an equivalence of abelian categories   
\begin{gather*}
{\rm p}^*:{\rm Coh}^{(B,\Psi)}(A\backslash B)\xrightarrow{\cong}{\rm Coh}^{(A\times B,\xi^{-1}\times \Psi)}(B),\\
(S,\rho)\mapsto \left(\O(B)\ot_{\O(A\backslash B)}S,\rho^*\right)
\end{gather*}
(see (\ref{rho*}) below), whose inverse is given by
\begin{gather*}
{\rm p}_*^{\left(A,\xi^{-1}\right)}:{\rm Coh}^{(A\times B,\xi^{-1}\times \Psi)}(B)\xrightarrow{\cong}{\rm Coh}^{(B,\Psi)}(A\backslash B),\\
(S,\rho)\mapsto \left({\rm p}_* ^{\left(A,\xi^{-1}\right)}S,\rho_2\right).
\end{gather*}
\end{theorem}

\begin{proof}
Let $(S,\rho)$ be any object in ${\rm Coh}^{(B,\Psi)}(A\backslash B)$. Consider the sheaf ${\rm p}^*S=\O(B)\ot_{\O(A\backslash B)}S$, and map
\begin{equation*}
\rho^*:=\left((\widetilde{\rho_{\Psi^{-1}}}\ot\id)\Delta\right)\bar{\ot}\rho
\end{equation*}
(see (\ref{tilderhoPsi})). 
Namely, we have
\begin{equation}\label{rho*}
\begin{split}
& \rho^*:{\rm p}^*S\to {\rm p}^*S\ot\O(A)_{\xi^{-1}}\ot \O(B)_{\Psi},\\
& f\ot_{\O(A\backslash B)}s\mapsto \Psi^{2} f_2\ot_{\O(A\backslash B)} s^0\ot \partial^{\sharp}{\rm S}\left(\Psi^{1}f_1\right)Q_{\xi^{-1}}\ot f_3s^1.
\end{split}
\end{equation}
It is straightforward to verify that $\left({\rm p}^*S,\rho^*\right)\in {\rm Coh}^{(A\times B,\xi^{-1}\times \Psi)}(B)$, so we have a functor 
$${\rm p}^*:{\rm Coh}^{(B,\Psi)}(A\backslash B)\xrightarrow{}{\rm Coh}^{(A\times B,\xi^{-1}\times \Psi)}(B),\,\,\,(S,\rho)\mapsto \left({\rm p}^*S,\rho^*\right).$$

Conversely, let $(S,\rho)\in {\rm Coh}^{(A\times B,\xi^{-1}\times \Psi)}(B)$. Then by Lemma \ref{inducedinvariantsquotient}, we have a functor 
$${\rm p}_*^{\left(A,\xi^{-1}\right)}:{\rm Coh}^{(A\times B,\xi^{-1}\times \Psi)}(B)\xrightarrow{}{\rm Coh}^{(B,\Psi)}(A\backslash B),\,\,(S,\rho)\mapsto \left({\rm p}_* ^{\left(A,\xi^{-1}\right)}S,\rho_2\right).
$$

Finally, it is straightforward to verify that the two functors constructed above are inverse to each other.  
\end{proof}

\begin{remark}\label{affinegrschs2}
If $B$ is any affine group scheme such that the quotient scheme  $A\backslash B$ is affine, then Theorem \ref{lemmatezer} and its proof hold after replacing ${\rm Coh}$ by ${\rm Coh}_{{\rm f}}$ (see \cite{G}).
\end{remark}

\subsection{The first equivalence $\Rep_k(A,\xi^{-1}) \cong {\rm Coh}^{(B,\Psi)}(A\backslash B)$}\label{sec:The first equivalence}
Set
\begin{equation}\label{indcomod}
{\rm Ind}_{(A,\xi^{-1})}^{(B,\Psi)}:={\rm p}{}^{(A,\xi^{-1})}_*{\rm t}^*,\,\,\,{\rm and}\,\,\,
{\rm Res}_{(A,\xi^{-1})}^{(B,\Psi)}:=\rm t_*^{(B,\Psi)}{\rm p}^*.
\end{equation} 

\begin{theorem}\label{modrep0-1}
The equivalences   
\begin{gather*}
{\rm Ind}_{(A,\xi^{-1})}^{(B,\Psi)}:{\rm Rep}_k (A,\xi^{-1})\xrightarrow{\cong}{\rm Coh}^{(B,\Psi)}\left(A\backslash B\right),\\
V\mapsto \left({\rm p}_*^{(A,\xi^{-1})}\left(\O(B)\ot_k V\right),\rho^{V}_2\right),
\end{gather*} 
and
\begin{gather*}
{\rm Res}_{(A,\xi^{-1})}^{(B,\Psi)}:{\rm Coh}^{(B,\Psi)}\left(A\backslash B\right)\xrightarrow{\cong}{\rm Rep}_k (A,\xi^{-1}),\\
(S,\rho)\mapsto \left(\left(\O(B)\ot_{\O(A\backslash B)}S\right)^{(B,\Psi)},\rho^*_3\right)
\end{gather*}
(see (\ref{natural structure}), (\ref{rho*})), are inverse to each other.
\end{theorem}

\begin{proof}
Follows from Theorems \ref{lemmafromg} and \ref{lemmatezer}. 
\end{proof}

Recall (\ref{fromlefttoright}) that for any {\em right} $\O(A)_{\xi^{-1}}$-comodule $(V,r)$, the map 
$$\widetilde{r}:V\to \O(A)_{\xi}\ot V,\,\,\,v\mapsto {\rm S}(v^1)Q_{\xi}\ot v^0,$$ 
defines a {\em left} $\O(A)_{\xi}$-comodule structure on $V$.

View $\O(B)$ as a right $\O(A)_{\xi}$-comodule via $\widetilde{\rho_{\Psi}}$ (\ref{tilderhoPsi}) given by 
$$\widetilde{\rho_{\Psi}}(f)=\Psi^{-2}f_2\ot \partial^{\sharp}{\rm S}\left(\Psi^{-1}f_1\right)Q_{\xi};\,\,\,f\in\O(B).$$

Recall that $\O(B)\ot^{\O(A)_{\xi}}V={\rm Ker}\left(\widetilde{\rho_{\Psi}}\ot {\rm id}-{\rm id}\ot \widetilde{r}\right)$.

\begin{proposition}\label{modrep00lemma-1}
For any right $\O(A)_{\xi^{-1}}$-comodule $(V,r)$, we have   
$$\O(B)\ot^{\O(A)_{\xi}}V=\left(\O(B)\ot_k V\right)^{\left(A,\xi^{-1}\right)}.$$
\end{proposition}

\begin{proof}
By (\ref{natural structure}), $\O(A)_{\xi^{-1}}$ coacts on $\O(B)\ot_k V$ via
$$f\ot v\mapsto \Psi^1_2 f_2\ot v^{0}\ot \partial^{\sharp}{\rm S}\left(\Psi^1_1f_1\right)v^{1}\varepsilon(\Psi^2 f_3)=f_2\ot v^0\ot \partial^{\sharp}{\rm S}(f_1)v^1.$$
Thus, by definition, $\sum_i f_i\ot v_i\in \left(\O(B)\ot_k V\right)^{\left(A,\xi^{-1}\right)}$ if and only if
\begin{equation}\label{coinvcoten}
\sum_i f_{i2}\ot v_i^0\ot \partial^{\sharp}{\rm S}(f_{i1})v_i^1=\sum_i \Psi^{2}f_i\ot v_i\ot \partial^{\sharp}{\rm S}(\Psi^{1})Q_{\xi^{-1}}.
\end{equation}

Now assume that $\sum_i f_i\ot v_i$ belongs to $\left(\O(B)\ot_k V\right)^{\left(A,\xi^{-1}\right)}$. We have to verify that $\sum_i f_i\ot v_i$ lies in $\O(B)\ot^{\O(A)_{\xi}}V$, i.e., that  
$$\sum_i \widetilde{\rho_{\Psi}}(f_i)\ot v_i=\sum_i f_i\ot \widetilde{r}(v_i).$$
Thus, we have to show that 
\begin{equation}\label{coinvcotenwant}
\sum_i \Psi^{-2}f_{i2}\ot \partial^{\sharp}{\rm S}(\Psi^{-1}f_{i1})Q_{\xi}\ot v_i=\sum_i f_i\ot {\rm S}(v_i^1)Q_{\xi}\ot v_i^0.
\end{equation}
To this end,
apply $(\id\ot\widetilde{r}\ot\id)$ to (\ref{coinvcoten}) to obtain
$$\sum_i f_{i2}\ot {\rm S}(v_i^1)Q_{\xi}\ot v_i^0\ot \partial^{\sharp}{\rm S}(f_{i1})v_i^2=\sum_i \Psi^{2}f_i\ot {\rm S}(v_i^1)Q_{\xi}\ot v_i^0\ot \partial^{\sharp}{\rm S}(\Psi^{1})Q_{\xi^{-1}}.$$
Thus, we have
$$\sum_i f_{i2}\ot {\rm S}(v_i^1)\ot v_i^0\ot \partial^{\sharp}{\rm S}(f_{i1})v_i^2=\sum_i \Psi^{2}f_i\ot {\rm S}(v_i^1)\ot v_i^0\ot \partial^{\sharp}{\rm S}(\Psi^{1})Q_{\xi^{-1}}.$$
Multiplying the second and fourth factors, yields
$$\sum_i f_{i2}\ot \partial^{\sharp}{\rm S}(f_{i1}){\rm S}(v_i^1)v_i^2\ot v_i^0=\sum_i \Psi^{2}f_i\ot {\rm S}(v_i^1)\partial^{\sharp}{\rm S}(\Psi^{1})Q_{\xi^{-1}}\ot v_i^0,$$
or equivalently (since ${\rm S}(v_i^1)v_i^2=\varepsilon(v_i^1)Q_{\xi^{-1}}$),
$$\sum_i f_{i2}\ot \partial^{\sharp}{\rm S}(f_{i1})\varepsilon(v_i^1)Q_{\xi^{-1}}\ot v_i^0=\sum_i \Psi^{2}f_i\ot {\rm S}(v_i^1)\partial^{\sharp}{\rm S}(\Psi^{1})Q_{\xi^{-1}}\ot v_i^0.$$
Thus, we have
$$\sum_i f_{i2}\ot \partial^{\sharp}{\rm S}(f_{i1})\ot v_i=\sum_i \Psi^{2}f_i\ot {\rm S}(v_i^1)\partial^{\sharp}{\rm S}(\Psi^{1})\ot v_i^0,$$
which is equivalent to (\ref{coinvcotenwant}).

Similarly, if $\sum_i f_i\ot v_i\in \O(B)\ot^{\O(A)_{\xi}}V$, then $\sum_i f_i\ot v_i$ lies in $\left(\O(B)\ot_k V\right)^{\left(A,\xi^{-1}\right)}$.
\end{proof}

\begin{theorem}\label{modrep0-1new}
The equivalences   
\begin{gather*}
{\rm Ind}_{(A,\xi^{-1})}^{(B,\Psi)}:{\rm Rep}_k (A,\xi^{-1})\xrightarrow{\cong}{\rm Coh}^{(B,\Psi)}\left(A\backslash B\right),\\
V\mapsto \left(\O(B)\ot^{\O(A)_{\xi}}V,\rho^{V}_2\right),
\end{gather*} 
and
\begin{gather*}
{\rm Res}_{(A,\xi^{-1})}^{(B,\Psi)}:{\rm Coh}^{(B,\Psi)}\left(A\backslash B\right)\xrightarrow{\cong}{\rm Rep}_k (A,\xi^{-1}),\\
(S,\rho)\mapsto \left(\left(\O(B)\ot_{\O(A\backslash B)}S\right)^{(B,\Psi)},\rho_3^*\right),
\end{gather*}
are inverse to each other.
\end{theorem}

\begin{proof}
Follows from Theorem \ref{modrep0-1} and Proposition \ref{modrep00lemma-1}.
\end{proof}

\begin{example}\label{regrepind}
There is a canonical isomorphism 
$${\rm Ind}_{(A,\xi^{-1})}^{(B,\Psi)}\left(\O(A),\Delta_{\xi^{-1}}\right)\cong \left(\O(B),\Delta_{\Psi}\right)$$
in ${\rm Coh}^{(B,\Psi)}\left(A\backslash B\right)$. \qed
\end{example}

\subsection{The second equivalence $\Rep_k(A,\xi^{-1}) \cong {\rm Coh}^{(B,\Psi)}(A\backslash B)$}\label{sec:thesecondequivalence}
{\em Choose} a cleaving map (\ref{nbpgammatilde}) $\widetilde{\mathfrak{c}}:\O(A)\xrightarrow{1:1} \O(B)$. Recall (\ref{nbpalphatilde0}), and let 
$$
\widetilde{\alpha}:\O(B)\twoheadrightarrow \O(A\backslash B),\,\,\,f\mapsto f_1\widetilde{\mathfrak{c}}^{-1}(\partial^{\sharp}(f_2)).
$$
Also, for any $V\in \Rep_k(A,\xi^{-1})$, consider the $k$-linear isomorphism   
\begin{gather*}
{\rm F}_V:=\widetilde{\alpha}\ot \id:\O(B)\ot^{\O(A)_{\xi}}V\xrightarrow{\cong} \O(A\backslash B)\ot_k V,\\
f\ot v\mapsto f_1\widetilde{\mathfrak{c}}^{-1}(\partial^{\sharp}(f_2))\ot v,
\end{gather*}
whose inverse is given by
\begin{gather*}
{\rm F}_V^{-1}:\O(A\backslash B)\ot_k V\xrightarrow{\cong}  \O(B)\ot^{\O(A)_{\xi}}V,\\
f\ot v\mapsto f\widetilde{\mathfrak{c}}(v^1)\ot v^0,
\end{gather*}
and the map 
$$\rho_{V}:=({\rm F}_V\ot\id)(23)(\Delta_{\Psi}\ot\id){\rm F}_V^{-1}.$$
Note that we have 
\begin{equation}\label{nuviexp}
\begin{split}
& \rho_V:\O(A\backslash B)\ot_k V\to \O(A\backslash B)\ot_k V\ot \O(B)_{\Psi},\\
& f\ot v\mapsto \Psi_1^1f_1\widetilde{\mathfrak{c}}(v^1)_1\widetilde{\mathfrak{c}}^{-1}\partial^{\sharp}\left(\Psi_2^1f_2\widetilde{\mathfrak{c}}(v^1)_2\right)\ot v^0\ot \Psi^2f_3\widetilde{\mathfrak{c}}(v^1)_3.
\end{split}
\end{equation} 

\begin{theorem}\label{modrep00}
We have an equivalence of abelian categories   
$${\rm F}:\Rep_k(A,\xi^{-1})\xrightarrow{\cong}{\rm Coh}^{(B,\Psi)}(A\backslash B),\,\,\,V\mapsto \left(\O(A\backslash B)\ot_k V,\rho_V\right).$$
\end{theorem}

\begin{proof}
Follows from Theorem \ref{modrep0-1} and the preceding remarks. 
\end{proof}

\section{Double cosets and biequivariant sheaves}\label{sec:doublecosetsandbiequivariant}
Recall that double cosets in finite abstract groups play a fundamental role in the theory of group-theoretical fusion categories (see, e.g, \cite{O}).

Let $\iota_H:H\xrightarrow{1:1}G$ and $\iota_K:K\xrightarrow{1:1}G$ be two embeddings of finite group schemes, and consider the right action of $H\times K$ on $G$ given by 
\begin{equation}\label{hkaction}
\mu_{G\times (H\times K)}:G\times (H\times K)\to G,\,\,\,(g,h,k)\mapsto h^{-1}gk. 
\end{equation}
The algebra map $\mu_{G\times (H\times K)}^{\sharp}:\O(G)\to \O(G)\ot \O(H\times K)$ is given by
\begin{equation}\label{hkactioncom}
\mu_{G\times (H\times K)}^{\sharp}(f)=f_2\ot \iota_H^{\sharp}{\rm S}(f_1)\ot \iota_K^{\sharp}(f_3);\,\,\,f\in\O(G). 
\end{equation}
Recall that since $H\times K$ is a finite group scheme, there exists a finite quotient scheme $Y:=G/(H\times K)$ (see  \S\ref{sec:quotmorphind}).

For any closed point $g\in G(k)$, let $Z_g:=HgK\subset G$ denote the orbit of $g$ under the action (\ref{hkaction}) -- it is a closed subset of $G$. It is clear that $Z_g\in Y(k)$, and all closed points of $Y$ are such. Letting 
${\rm Coh}(Y)_{Z_g}\subset {\rm Coh}(Y)$ denote the abelian subcategory of sheaves on $Y$ supported on $Z_g$ (see \S\ref{sec:coh decomposition}), we have a direct sum decomposition of abelian categories
\begin{equation}\label{doublecosetfadec}
{\rm Coh}(Y)=\bigoplus_{Z_g\in Y(k)}{\rm Coh}(Y)_{Z_g}.
\end{equation}

For the remaining of this section, we fix $Z\in Y(k)$ with representative $g\in Z(k)$.   
Let $L^g:=H\cap gKg^{-1}$ denote the (group scheme-theoretical) intersection of $H$ and $gKg^{-1}$, i.e., $L^g=H\times_{G}\, gKg^{-1}$. Clearly, $L^g$ is a closed subgroup scheme of both $H$ and $gKg^{-1}$, and  
$$\O(L^g)=\O(H)\otimes_{\O(G)} \O(gKg^{-1})=\O(G)/\left(\mathscr{I}(H)+\mathscr{I}(gKg^{-1})\right).$$
Let $\iota_{g}:L^g\hookrightarrow G$ be the inclusion morphism, and consider the group scheme embedding of $L^g$ in $H\times K$ given by 
\begin{equation}\label{thetag}
\partial_g:=\left(\iota_{g},\left({\rm Ad}g^{-1}\right)\circ\iota_{g}\right):L^g\xrightarrow{1:1} H\times K,\;\;\ell\mapsto (\ell,g^{-1}\ell g).
\end{equation}  
It is clear that the subgroup scheme $\partial_g(L^g)\subset H\times K$ is the stabilizer of $g$ for the action (\ref{hkaction}). Namely, 
$Z$ is a quotient scheme for the free {\em left} action of $L^g$ on $H\times K$ given by 
\begin{equation}\label{lgactsonhk}
\mu_{L^g\times (H\times K)}:L^g\times (H\times K)\to H\times K,\,\,\,(\ell,h,k)\mapsto (\ell h,g^{-1}\ell gk).
\end{equation}  
Thus, we have scheme isomorphisms  
\begin{equation}\label{doucosetsch}
\begin{split}
& \mathfrak{j}_{g}:L^g\backslash (H\times K)\xrightarrow{\cong }Z,\,\,\,\overline{(h,k)}\mapsto h^{-1}gk,\,\,\,{\rm and}\\
& \mathfrak{j}_{g}^{-1}:Z\xrightarrow{\cong }L^g\backslash (H\times K),\,\,\,hgk\mapsto \overline{(h^{-1},k)}.
\end{split}
\end{equation}

Assume that $\psi\in Z^2(H,\mathbb{G}_m)$ and $\eta\in Z^2(K,\mathbb{G}_m)$. Set  
\begin{equation}\label{xig0}
\xi_g:=\left(\partial_g^{\sharp}\ot \partial_g^{\sharp}\right)\left(\psi^{-1}\times \eta\right)=\psi^{-1}\eta^{g^{-1}}\in Z^2(L^g,\mathbb{G}_m).
\end{equation}

Next {\em choose} a cleaving map (\ref{nbpgammatilde}) $\widetilde{\mathfrak{c}_g}:\O(L^g)\xrightarrow{1:1} \O(H\times K)$, and let
\begin{equation}\label{cCalphaCpredouble1}
\widetilde{\alpha_g}:\O(H\times K)\twoheadrightarrow\O(L^g\backslash (H\times K)),\,\,\,{\rm f}\mapsto {\rm f}_1\widetilde{\mathfrak{c}_g}^{-1}\left(\partial_{g}^{\sharp}\left({\rm f}_2\right) \right)
\end{equation}
(see (\ref{nbpalphatilde0})). Set (see Theorems \ref{modrep0-1new}, \ref{modrep00})
\begin{equation}\label{boldfzg1}
{\rm Ind}_{Z}:={\rm Ind}_{(L^g,\xi_g^{-1})}^{\left(H\times K,\psi^{-1}\times \eta\right)},\,\,\,{\rm and}\,\,\,{\rm F}_{Z}:={\rm F}.
\end{equation}

\begin{corollary}\label{newimpcor}
For any closed point $Z\in Y(k)$ with representative $g\in Z(k)$, we have the following equivalences of abelian categories:
\begin{enumerate}
\item
${\rm Ind}_{Z}:{\rm Rep}_k (L^g,\xi_g^{-1})\xrightarrow{\cong}{\rm Coh}^{\left(H\times K,\psi^{-1}\times \eta\right)}(L^g\backslash (H\times K)),$\\
$V\mapsto \left(\O(H\times K)\ot^{\O(L^g)_{\xi_g}}V,\rho_{2}^V\right)$.
\item
${\rm F}_{Z}:{\rm Rep}_k (L^g,\xi_g^{-1})\xrightarrow{\cong}{\rm Coh}^{\left(H\times K,\psi^{-1}\times \eta\right)}(L^g\backslash (H\times K)),$\\
$V\mapsto (\O(L^g\backslash (H\times K))\ot_k V,\rho_V)$.
\item
$\mathfrak{j}_{g*}:{\rm Coh}^{\left(H\times K,\psi^{-1}\times \eta\right)}(L^g\backslash (H\times K))\xrightarrow{\cong}{\rm Coh}^{\left(H\times K,\psi^{-1}\times \eta\right)}(Z)$.
\end{enumerate}
\end{corollary}

\begin{proof}
Follow from Theorems \ref{modrep0-1new}, \ref{modrep00} for $A:=L^g$, $B:=H\times K$, $\partial:=\partial_g$, $\Psi:=\psi^{-1}\times \eta$, and $\xi:=\xi_g$.
\end{proof}

\section{Module categories over $\C\left(G,H,\psi\right)$}\label{sec:Module categories over GCTC}

Fix a finite group scheme-theoretical category $\C:=\C\left(G,H,\psi\right)$ as in \S\ref{sec:M(H,psi)}. In this section we study the abelian structure of indecomposable exact $\C$-module categories, which can be classified as categories of ${\rm Coh}(G)$-module functors \cite{G}.

\begin{theorem}\cite[Theorem 5.7]{G}\label{modgth}
The assignment 
$$\mathscr{M}(K,\eta)\mapsto \Fun_{{\rm Coh}(G)}\left(\mathscr{M}(H,\psi),\mathscr{M}(K,\eta)\right)$$
determines an equivalence between the $2$-category of indecomposable exact module categories over ${\rm Coh}(G)$ and
the $2$-category of indecomposable exact module categories over
$\C(G,H,\psi)$. \qed
\end{theorem}

\subsection{Three descriptions of module categories} 
For $H, K, \psi, \eta$ as above, consider the category 
\begin{equation}\label{defnmodcatM}
\mathscr{M}\left(H\times K,\psi^{-1}\times \eta\right)={\rm Coh}^{\left(H\times K,\psi^{-1}\times \eta\right)}(G)
\end{equation}
of $\left(H\times K,\psi^{-1}\times \eta\right)$-equivariant  sheaves on $G$ with respect to the right action $\mu_{G\times (H\times K)}$ from \eqref{hkaction}.
Note that there is an equivalences of abelian categories
$$
\M\left(H\times K,\psi^{-1}\times \eta\right)\cong {\rm Bicomod}_{{\rm Coh}(G)}\left(\mathscr{O}(H)_{\psi},\mathscr{O}(K)_{\eta}\right).
$$

\begin{lemma}\label{newsimple}
There is an equivalence of abelian categories
$$\Fun_{{\rm Coh}(G)}\left(\mathscr{M}(H,\psi),\mathscr{M}(K,\eta)\right)\cong \M\left(H\times K,\psi^{-1}\times \eta\right).$$
\end{lemma}

\begin{proof}
A functor $\mathscr{M}(H,\psi)\to\mathscr{M}(K,\eta)$ is
determined by a $(K,\eta)$-equivariant  sheaf $S$ on $G$ (its value on $\O(H)_{\psi}$), and the fact that the functor is a ${\rm Coh}(G)$-module (hence, ${\rm Coh}(H)$-module) functor gives $S$ a commuting $H$-equivariant structure for the right action $\widetilde{\mu_{G\times H}}$ of $H$ on $G$, i.e., $S\in \M\left(H\times K,\psi^{-1}\times \eta\right)$.
 
Conversely, it is clear that any $S\in \M\left(H\times K,\psi^{-1}\times \eta\right)$, viewed as an object in ${\rm Bicomod}_{{\rm Coh}(G)}\left(\mathscr{O}(H)_{\psi},\O(K)_{\eta}\right)$, defines a ${\rm Coh}(G)$-module functor 
$$\mathscr{M}(H,\psi)\to\mathscr{M}(K,\eta),\,\,\,T\mapsto T\ot^{\O(H)_{\psi}} S,$$ 
where $T\ot^{\O(H)_{\psi}} S$ is the cotensor product.
\end{proof}

\subsection{The simple objects of $\mathscr{M}\left(H\times K,\psi^{-1}\times \eta\right)$}\label{sec:The simple objects} Fix an indecomposable exact left $\C$-module category 
$\M:=\mathscr{M}\left(H\times K,\psi^{-1}\times \eta\right)$ (\ref{defnmodcatM}).

Fix a closed point $Z\in Y(k)$, and let  
\begin{equation}\label{defnofmzg}
\M_{Z}:=\M_{Z}\left(H\times K,\psi^{-1}\times \eta\right)\subset \M
\end{equation}
denote the full abelian subcategory of $\M$ consisting of all objects annihilated by the defining ideal $\mathscr{I}(Z)\subset \mathscr{O}(G)$ of $Z$. Namely, let 
\begin{equation*}
\iota_{Z}:Z\hookrightarrow G
\end{equation*}
denote the inclusion morphism (it is $H\times K$-equivariant). Then $\M_{Z}$ is the image of the injective functor 
\begin{equation}\label{mg}
\iota_{Z*}:{\rm Coh}^{\left(H\times K,\psi^{-1}\times \eta\right)}(Z)\xrightarrow{1:1}{\rm Coh}^{\left(H\times K,\psi^{-1}\times \eta\right)}(G).
\end{equation}

Let $\overline{\M_{Z}}\subset \M$ denote the Serre closure of $\M_{Z}$ inside $\M$, i.e., $\overline{\M_{Z}}$ is the full abelian subcategory of $\M$ consisting of all objects whose composition factors lie in $\M_{Z}$.

Fix a representative $g\in Z(k)$, and define the functor
\begin{equation}\label{Indz}
\mathbf{Ind}_{Z}:=\iota_{Z*}\circ\mathfrak{j}_{g*}\circ{\rm Ind}_{Z}.
\end{equation}

\begin{theorem}\label{simmodrep}
Let $\M=\mathscr{M}\left(H\times K,\psi^{-1}\times \eta\right)$ be as above.
\begin{enumerate}
\item
For any closed point $Z\in Y(k)$ with representative $g\in Z(k)$, we have an equivalence of abelian categories
\begin{gather*}
\mathbf{Ind}_{Z}:{\rm Rep}_k (L^g,\xi_g^{-1})\xrightarrow{\cong}\M_{Z},\\
V\mapsto \left(\iota_{Z}\mathfrak{j}_g\right)_*\left(\O(H\times K)\ot^{\O(L^g)_{\xi_g}}V,\rho_{2}^V\right).
\end{gather*}
\item
There is a bijection between equivalence classes of pairs $(Z,V)$, where $Z\in Y(k)$ is a closed point with representative $g\in Z(k)$, and $V\in{\rm Rep}_k (L^g,\xi_g^{-1})$ is simple, and simple objects of $\mathscr{M}$, assigning $(Z,V)$ to $\mathbf{Ind}_{Z}(V)$.
\item
We have a direct sum decomposition of abelian categories
$$\M=\bigoplus_{Z\in Y(k)}\overline{\M_{Z}}.$$
\end{enumerate}  
\end{theorem}

\begin{proof}
(1) Follows from Corollary \ref{newimpcor}.

(2) Assume $S\in \M$ is simple. Since for any $Z\in Y(k)$, $\mathscr{I}(Z)$ is a $H\times K$-stable ideal, it follows from Proposition \ref{abelcat}(3) that either $\mathscr{I}(Z)S=0$ or $\mathscr{I}(Z)S=S$. If $\mathscr{I}(Z)S=S$ for every $Z\in Y(k)$, then $(\Pi_{Z}\mathscr{I}(Z))S=S$. But, $\Pi_{Z}\mathscr{I}(Z)$ is a nilpotent ideal of $\O(G)$ (being contained in the radical of $\O(G)$), so $S=0$, a contradiction.

Thus, there exists $Z\in Y(k)$ such that $\mathscr{I}(Z)S=0$. Assume that $\mathscr{I}(Z')S=0$ for some $Z'\in Y(k)$, $Z'\ne Z$. Then $\O(G)S=0$ (since $\O(G)=\mathscr{I}(Z)+\mathscr{I}(Z')$), a contradiction. It follows that there exists a unique $Z\in Y(k)$ such that $\mathscr{I}(Z)S=0$, i.e., $S\in\M_{Z}$, so the claim follows from (1).

(3) Follows from (2), and the fact that there are no nontrivial cross extensions of $\O(G)$-modules which are supported on distinct closed points of $G$ (see \S\ref{sec:coh decomposition}).
\end{proof}

\subsection{Projectives in $\mathscr{M}\left(H\times K,\psi^{-1}\times \eta\right)$}\label{sec:The projective objects}
Let $\M$ be as in \S\ref{sec:The simple objects}.
 
For any $Z\in Y(k)$ with representative $g\in Z(k)$, recall the functor ${\rm F}_{Z}$ from Corollary \ref{newimpcor}, and define the functor
\begin{equation}\label{boldfzg1F}
\mathbf{F}_{Z}:=\iota_{Z*}\circ\mathfrak{j}_{g*}\circ{\rm F}_{Z}.
\end{equation}
Also, set $Z^{\circ}:= H^{\circ} g K^{\circ}$,   
\begin{equation}\label{boldfzg1pc}
|Z^{\circ}|:=\frac{|H^{\circ}||K^{\circ}|}{|(L^g)^{\circ}|}\in \mathbb{Z}^{\ge 1},\,\,\,{\rm and}\,\,\,
|Z(k)|:=\frac{|H(k)||K(k)|}{|(L^g)(k)|}\in \mathbb{Z}^{\ge 1}.
\end{equation}
Note that (\ref{ses0}) induces 
a split exact sequence of schemes
\begin{equation*}
1\to Z^{\circ}\xrightarrow{i_{Z^{\circ}}} Z \mathrel{\mathop{\rightleftarrows}^{\pi_{Z}}_{q_{Z}}} Z(k)\to 1.
\end{equation*} 

\begin{lemma}\label{closedptproj}
For any $Z\in Y(k)$, with representative closed point $g\in Z(k)$, the following hold:
\begin{enumerate}
\item
The surjective $\O(G)$-linear map
$$\theta_{Z}:=(\id\ot q_{Z}^{\sharp}):\O(G^{\circ})\ot \O(Z)\twoheadrightarrow \O(G^{\circ})\ot\O(Z(k))$$
splits via the map
$$\lambda_{Z}:=(\id\ot \pi_{Z}^{\sharp}):\O(G^{\circ})\ot\O(Z(k))\xrightarrow{1:1}\O(G^{\circ})\ot \O(Z).$$
Thus, $\O(G^{\circ})\ot \O(Z(k))$ is a direct summand of $ \O(G^{\circ})\ot \O(Z)$ as an $\O(G)$-module.
\item 
For any simple $V\in \Rep_k(L^g,\xi_g^{-1})$, we have an $\M$-isomorphism
$$\O(G^{\circ})\ot \mathbf{F}_{Z}\left(P_{(L^g,\xi_g^{-1})}(V)\right)\cong \O(Z^{\circ})\ot_kP_{\M}\left(\mathbf{F}_{Z}(V)\right).$$
Here, $\O(G)$ acts diagonally on $\O(G^{\circ})\ot \mathbf{F}_{Z}\left(P_{(L^g,\xi_g^{-1})}(V)\right)$, and $\O(H\times K)_{\psi^{-1}\times\eta}$ coacts on the second factor, 
and the right hand side is a direct sum of $|Z^{\circ}|$ copies of $P_{\M}\left(\mathbf{F}_{Z}(V)\right)$.
\end{enumerate}
\end{lemma}

\begin{proof}
(1) Follows from the preceding remarks.

(2) By Example \ref{regrepind}, $\mathbf{F}_{Z}\left(\O(L^g)_{\xi_g^{-1}}\right)\cong\O(H\times K)$ in $\M$, so $$\O(G)\ot \mathbf{F}_{Z}\left(\O(L^g)_{\xi_g^{-1}}\right)\cong \O(G)\ot \O(H\times K)$$ is projective in $\M$ (where $\O(G)$ acts diagonally, and $\O(H\times K)_{\psi^{-1}\times \eta}$ coacts on the second factor via $\Delta_{\psi^{-1}\times \eta}$). Since $\O(G^{\circ})$ is a direct summand of $\O(G)$, and $\mathbf{F}_{Z}\left(P_{(L^g,\xi_g^{-1})}(V)\right)$ is a direct summand of $\mathbf{F}_{Z}\left(\O(L^g)_{\xi_g^{-1}}\right)$, it follows that the object $\O(G^\circ)\ot \mathbf{F}_{Z}\left(P_{(L^g,\xi_g^{-1})}(V)\right)$ is a direct summand of $\O(G)\ot \O(H\times K)_{\psi^{-1}\times\eta}$, hence projective in $\M$. Thus, $\O(G^{\circ})\ot \mathbf{F}_{Z}\left(P_{(L^g,\xi_g^{-1})}(V)\right)$ is projective in $\M$.

Now, on the one hand, we have
\begin{eqnarray*}
\lefteqn{\Hom_{\M}\left(\O(G^\circ)\ot \mathbf{F}_{Z}\left(\O(L^g)_{\xi_g^{-1}}\right),\mathbf{F}_{Z}(V)\right)}\\
& = & \Hom_{\M}\left(\O(G^\circ)\ot \O(H\times K)_{\psi^{-1}\times\eta},\mathbf{F}_{Z}(V)\right)\\
& = & \Hom_{{\rm Coh}(G)}\left(\O(G^\circ),\O(Z)\right)\ot_k V.
\end{eqnarray*}

On the other hand, observe that for any simple $W\in \Rep_k(L^g,\xi_g^{-1})$, the objects $\O(G^\circ)\ot \mathbf{F}_{Z}\left(P_{(L^g,\xi_g^{-1})}(W)\right)$ and $\mathbf{F}_{Z}\left(P_{(L^g,\xi_g^{-1})}(W)\right)$ have the same composition factors as $\O(G)$-modules. Hence, if $V$ and $W$ are nonisomorphic simples in $\Rep_k(L^g,\xi_g^{-1})$, then 
$$\Hom_{\M}\left(\O(G^\circ)\ot \mathbf{F}_{Z}\left(P_{(L^g,\xi_g^{-1})}(W)\right),\mathbf{F}_{Z}(V)\right)=0.$$
This implies that  
\begin{eqnarray*}
	\lefteqn{\Hom_{\M}\left(\O(G^\circ)\ot \mathbf{F}_{Z}\left(\O(L^g)_{\xi_g^{-1}}\right),\mathbf{F}_{Z}(V)\right)}\\
	& = & \bigoplus_{W\in {\Irr}\left(\O(L^g)_{\xi_g^{-1}}\right)}\Hom_{\M}\left(\O(G^\circ)\ot \mathbf{F}_{Z}\left(P_{(L^g,\xi_g^{-1})}(W)\right),\mathbf{F}_{Z}(V)\right)\ot_k W\\
	& = & \Hom_{\M}\left(\O(G^\circ)\ot \mathbf{F}_{Z}\left(P_{(L^g,\xi_g^{-1})}(V)\right),\mathbf{F}_{Z}(V)\right)\ot_k V.
\end{eqnarray*}

Thus, it follows from the above that  
\begin{eqnarray*}
\lefteqn{\dim\Hom_{\M}\left(\O(G^\circ)\ot \mathbf{F}_{Z}\left(P_{(L^g,\xi_g^{-1})}(V)\right),\mathbf{F}_{Z}(V)\right)}\\
& = & \dim\Hom_{{\rm Coh}(G)}\left(\O(G^\circ),\O(Z)\right)=|Z^{\circ}|,
\end{eqnarray*}
which implies the statement.
\end{proof}

\begin{theorem}\label{prsimmodrep}
Let $\M=\M\left(H\times K,\psi^{-1}\times \eta\right)$ be as above.
\begin{enumerate}
\item
For any $Z\in Y(k)$ with representative $g\in Z(k)$, we have an equivalence of abelian categories
$$\mathbf{F}_{Z}:{\rm Rep}_k (L^g,\xi_g^{-1})\xrightarrow{\cong} \M_Z,\,\,\,
V\mapsto \iota_{Z*}\left(\O(Z)\ot_k V,\rho^g_V\right),$$
where, 
$\rho^g_V:=\left(\left(\mathfrak{j}_{g}^{\sharp}\right)^{-1}\ot\id^{\ot 3}\right)\rho_V\left(\mathfrak{j}_{g}^{\sharp}\ot\id\right)$ (\ref{nuviexp}).
\item
For any simple $V\in {\rm Rep}_k (L^g,\xi_g^{-1})$, we have
$$P_{\M}\left(\mathbf{F}_{Z}(V)\right)\cong \left(\O(G^{\circ})\ot \O(Z(k))\ot_k P_{(L^g,\xi_g^{-1})}(V),R^g_V\right),$$
where $R^g_{V}:=\left(\theta_{Z}\ot\id^{\ot 3}\right)\left(\id\ot\rho^g_{P_{(L^g,\xi_g^{-1})}(V)}\right)\left(\lambda_{Z}\ot\id\right)$.
\end{enumerate}  
\end{theorem}

\begin{proof}
(1) Follows from Theorem \ref{modrep00} and \ref{simmodrep}.

(2) We have $\M$-isomorphisms 
\begin{eqnarray*}
\lefteqn{\O(G^{\circ})\ot \mathbf{F}_{Z}\left(P_{(L^g,\xi_g^{-1})}(V)\right)= \O(G^{\circ})\ot \O(Z)\ot_k P_{(L^g,\xi_g^{-1})}(V)}\\
& \cong & \O(G^{\circ})\ot \O(Z(k))\ot_k \O(Z^{\circ})\ot_k P_{(L^g,\xi_g^{-1})}(V),
\end{eqnarray*}
thus by Proposition \ref{closedptproj}, $\O(G^{\circ})\ot \O(Z(k))\ot_k P_{(L^g,\xi_g^{-1})}(V)$ is a direct summand in a projective object of $\M$, hence is projective. Moreover, by Proposition \ref{closedptproj} again, we have  
$$\dim\Hom_{\M}\left(\O(G^{\circ})\ot \O(Z(k))\ot_k P_{(L^g,\xi_g^{-1})}(V),\mathbf{F}_{Z}(V)\right)=1.$$ 
Hence, $\O(G^{\circ})\ot \O(Z(k))\ot_k P_{(L^g,\xi_g^{-1})}(V)$ is the projective cover of $\mathbf{F}_{Z}(V)$ in $\M$, as claimed.
\end{proof}

\begin{example}\label{structcohhpsi}
Take $K=1$. Then for any $Z\in Y(k)$ with representative $g\in G(k)$, we have $Z=Hg$, $L^g=1$, and the functor 
$$\mathbf{F}_{Hg}:\Vect\xrightarrow{\cong}{\rm Coh}^{(H,\psi^{-1})}(Hg),\,\,\,
V\mapsto \O(Hg)\ot_k V,$$
is an equivalence of abelian categories (where $\O(H)$ acts on $\O(Hg)$ via multiplication by $\rho_{g^{-1}}(\cdot)$ and $\O(H)_{\psi^{-1}}$-coacts via $\Delta_{\psi^{-1}}$).

Now by Theorem \ref{prsimmodrep}, there is a bijection between the set of closed points $Hg\in Y(k)$ (i.e., isomorphism classes of simples of ${\rm Coh}(H\backslash G)$) and simple objects of $\mathscr{M}$, assigning $Hg$ to $\O(Hg)\ot_k k=\O(Hg)$, and we have $\overline{\M_{Hg}}=\langle \O(Hg)\rangle$.

Moreover, by Theorem \ref{prsimmodrep} again, for any simple $\O(Hg)$ in $\M$, 
$$P_{\M}(\O(Hg))\cong P_{Hg}\ot_k k=P_{Hg}$$
as $\O(G)$-modules, where $\O(H)_{\psi^{-1}}$ coacts on $P_{Hg}\ot_k k=P_{Hg}$ via the map $\rho_k=\Delta_{\psi^{-1}}$.
(Compare with Remark \ref{alsohelpful2}.) \qed
\end{example}

\begin{example}\label{doubletwisthopf}
Theorems \ref{simmodrep}, \ref{prsimmodrep} apply to the representation categories of two-sided twisted function (Hopf) algebras. Namely, let $I,J\in (\O(G)^{\ot 2})^*$ be two Hopf $2$-cocycles for $G$, and let ${}_I\mathscr{O}(G)_J$ be the corresponding two-sided twisted function algebra (if $I=J$, then ${}_J\mathscr{O}(G)_J$ is a Hopf algebra). By \cite[Corollary 4.8]{G}, $I,J$ correspond to pairs $(H,\psi)$, $(K,\eta)$, where $H,K\subset G$ are closed subgroup schemes and $\psi\in Z^2(H,\mathbb{G}_m)$, $\eta\in Z^2(K,\mathbb{G}_m)$ are {\em nondegenerate} $2$-cocycles (i.e., ${\rm Rep}_k(H,\psi)\cong \Vect$ and ${\rm Rep}_k(K,\eta)\cong \Vect$). Then it is straightforward to verify that 
we have an equivalence of abelian categories
$${\rm Rep}_k({}_I\mathscr{O}(G)_J)\cong \mathscr{M}\left(H\times K,\psi^{-1}\times\eta\right).$$
In particular, we have an equivalence of tensor categories
$${\rm Rep}_k({}_J\mathscr{O}(G)_J)\cong \C(G,H,\psi).$$
(In the fusion case it was proved in \cite{AEGN}.) \qed
\end{example}

\subsection{Fiber functors on $\C(G,H,\psi)$} Recall that a fiber functor on a finite tensor category $\C$ is the same as a $\C$-module category of rank $1$.

The following corollary is known in the fusion case \cite{O}.

\begin{corollary}\label{fibfungth}
Let $\C:=\C\left(G,H,\psi\right)$ be a group scheme-theoretical category. There is
a bijection between equivalence classes of fiber functors on $\C$ and equivalence classes of pairs $(K,\eta)$, where
$K\subset G$ is a closed subgroup scheme and $\eta\in
Z^2(K,\mathbb{G}_m)$, such that $HK=G$ and the $2$-cocycle $\xi^{-1}:=\xi_1^{-1}\in Z^2(H\cap K,\mathbb{G}_m)$  
is nondegenerate. 
\end{corollary}

\begin{proof}
Let $\M:=\mathscr{M}\left(H\times K,\psi^{-1}\times \eta\right)$ be an indecomposable exact module category over $\C$. 
By Theorem \ref{simmodrep}, $\M\cong \Vect$ if and only if  
$\M=\M_{HK}$ and ${\rm Rep}(H\cap K,\xi^{-1})\cong \Vect$. Thus, the statement follows from the fact that $\M=\M_{HK}$ if and only if $\iota_{HK*}$ is an equivalence, i.e., if and only if $G=HK$.
\end{proof}

\begin{example}\label{exacfac}
Any exact factorization $G=HK$ of finite group schemes \cite{BG} determines a fiber functor on $\C\left(G,H,\psi\right)$ given by $\M(K,1)$.
\end{example}

\section{The structure of $\C\left(G,H,\psi\right)$}\label{S:structure-GSC}

Fix a group scheme-theoretical category $\C:=\C\left(G,H,\psi\right)$ \S\ref{sec:M(H,psi)}. For any closed point $g\in G(k)$, let $H^g:=H\cap gHg^{-1}$. Note that $\xi_1=1$.

\begin{theorem}\label{simpobjs}  
The following hold:
\begin{enumerate}
\item
For any closed point $Z\in Y(k)$ with representative $g\in Z(k)$, we have an equivalence of abelian categories 
\begin{gather*}
\mathbf{Ind}_{Z}:{\rm Rep}_k (H^g,\xi_g^{-1})\xrightarrow{\cong} \mathscr{C}_{Z},\\
V\mapsto (\iota_{Z}\mathfrak{j}_g)_*\left(\O(H\times H)\ot^{\O(H^g)_{\xi_g}}V,\rho_2^{V}\right).
\end{gather*}

In particular,     
\begin{gather*}
\mathbf{Ind}_{H}:{\rm Rep}_k (H)\xrightarrow{\cong} \mathscr{C}_{H},\\
V\mapsto (\iota_{H}\mathfrak{j}_1)_*\left(\O(H\times H)\ot^{\O(H)}V,\rho_2^{V}\right),
\end{gather*}
is an equivalence of tensor categories.
\item
There is a bijection between equivalence classes of pairs $(Z,V)$, where $Z\in Y(k)$ is a closed point with representative $g\in Z(k)$, and $V\in{\rm Rep}_k (H^g,\xi_g^{-1})$ is simple, and simple objects of $\C$, assigning $(Z,V)$ to $\mathbf{Ind}_{Z}(V)$. Moreover, we have a direct sum decomposition of abelian categories
$$\C=\bigoplus_{Z\in Y(k)}\overline{\C_{Z}},$$
and 
$\overline{\C_{H}}\subset \C$ is a tensor subcategory. 
\item
For any $V\in {\rm Rep}_k (H^g,\xi_g^{-1})$, we have
$$
{\rm FPdim}\left(\mathbf{Ind}_{Z}(V)\right)=\frac{|H|}{|H^g|}{\rm dim}(V).$$
\item
For any $V\in {\rm Rep}_k (H^g,\xi_g^{-1})$, we have
$
\mathbf{Ind}_{Z}(V)^*\cong \mathbf{Ind}_{Z^{-1}}(V^*)$, 
where $Z^{-1}\in Y(k)$ such that $g^{-1}\in Z^{-1}(k)$.
\end{enumerate}
\end{theorem}

\begin{proof}
Follow from Theorem \ref{simmodrep}.
\end{proof}

\begin{theorem}\label{projsimpobjs}  
The following hold:
\begin{enumerate}
\item
For any closed point $Z\in Y(k)$ with representative $g\in Z(k)$, we have an equivalence of abelian categories
$$\mathbf{F}_{Z}:{\rm Rep}_k (H^g,\xi_g^{-1})\xrightarrow{\cong} \C_Z,\,\,\,
V\mapsto \iota_{Z*}\left(\O(Z)\ot_k V,\rho^g_V\right).$$

In particular, 
we have an equivalence of tensor categories   
$$\mathbf{F}_H:{\rm Rep}_k (H)\xrightarrow{\cong} \mathscr{C}_{H},\,\,\,V\mapsto \iota_{H*}\left(\O(H)\ot_k V,\rho^1_V\right).$$
\item
For any $V\in {\rm Rep}_k (H^g,\xi_g^{-1})$, we have
$
\mathbf{F}_{Z}(V)^*\cong \mathbf{F}_{Z^{-1}}(V^*)$.
\item
For any simple $V\in {\rm Rep}_k (H^g,\xi_g^{-1})$, we have 
$$P_{\C}\left(\mathbf{F}_{Z}(V)\right)\cong 
\left(\O(G^{\circ})\ot \O(Z(k))\ot_k P_{(H^g,\xi_g^{-1})}(V),R^g_V\right).$$
In particular, 
$${\rm FPdim}\left(P_{\C}\left(\mathbf{F}_{Z}\left(V\right)\right)\right)=\frac{|G^{\circ}||H(k)|}{|H^{\circ}||H^g(k)|}{\rm dim}\left(P_{(H^g,\xi_g^{-1})}\left(V\right)\right).$$
\item
For any closed point $Z\in Y(k)$, 
${\rm FPdim}\left(\overline{\C_{Z}}\right)=|G^{\circ}||Z(k)|$.
\item
If $\Rep_k(H)$ is unimodular, so is $\C$ (but not necessarily vice versa).
\end{enumerate}
\end{theorem}

\begin{proof}
(1)-(4) Follow from Theorem \ref{prsimmodrep} in a straightforward manner.

(5) Recall that $\Rep_k(H)$ is unimodular if and only if $P(\mathbf{1})\cong P(\mathbf{1})^*$, where $P(\mathbf{1})$ is the projective cover of $\mathbf{1}$ in $\Rep_k(H)$. Thus, if $\Rep_k(H)$ is unimodular then 
\begin{eqnarray*}
\lefteqn{|H^{\circ}|P_{\C}(\mathbf{1})\cong\O(G^{\circ})\ot \mathbf{F}\left(P(\mathbf{1})\right)}\\
& \cong & \O(G^{\circ})\ot \mathbf{F}\left(P(\mathbf{1})^*\right)\cong \O(G^{\circ})\ot \mathbf{F}\left(P(\mathbf{1})\right)^*\\
& \cong & \O(G^{\circ})^*\ot \mathbf{F}\left(P(\mathbf{1})\right)^*\cong \left(\O(G^{\circ})\ot \mathbf{F}\left(P(\mathbf{1})\right)\right)^*\\
& \cong & |H^{\circ}| P_{\C}(\mathbf{1})^*,
\end{eqnarray*}
so $P_{\C}(\mathbf{1})\cong P_{\C}(\mathbf{1})^*$. 
\end{proof}

\begin{remark}
Theorem \ref{projsimpobjs}(5) for \'{e}tale groups follows from \cite{Y}.
\end{remark}

\subsection{The \'{e}tale case}\label{sec:The etale case}
Assume $G$ is {\em \'{e}tale}, i.e, $G=G(k)$. The following result is known in characteristic $0$ \cite{GN,O}.

\begin{corollary}\label{projcoveretaleg}
The following hold:
\begin{enumerate}
\item
For any $(H,H)$-double coset $Z$ with representative $g\in G$, and simple $V\in {\rm Rep}_k (H^g,\xi_g^{-1})$, we have 
$$P_{\C}\left(\mathbf{F}_{Z}(V)\right)=\mathbf{F}_{Z}\left(P_{(H^g,\xi_g^{-1})}\left(V\right)\right)=\left(\O\left(Z\right)\ot_k P_{(H^g,\xi_g^{-1})}\left(V\right),R^g_V\right).$$
\item
We have a direct sum decomposition of abelian categories
$$\C\cong \bigoplus_{Z\in Y}{\rm Rep}_k (H^g,\xi_g^{-1}),$$
i.e., $\overline{\C_{Z}}=\C_{Z}$ for every $Z$.
\item
$\C$ is fusion if and only if $p$ does not divide $|H|$.
\end{enumerate}
\end{corollary}

\begin{proof}
Follow immediately from Theorem \ref{projsimpobjs}.
\end{proof}

\subsection{The normal case}\label{sec:The normal case} 
Assume $H$ is normal in $G$. Then $G/H=H\backslash G$ is a finite group scheme, and the quotient morphism $\pi:G\to G/H$ (\ref{quotient morphism}) is a group scheme morphism. 

Recall (\ref{yetanotherrhoU}) the maps $\rho^{\psi}_U$, $U\in {\rm Coh}(G/H)$.

\begin{theorem}\label{projcovernormalg}
The following hold:
\begin{enumerate}
\item 
The injective tensor functor $\pi^*:{\rm Coh}(G/H)\xrightarrow{1:1}{\rm Coh}(G)$ lifts 
to an injective tensor functor 
$$\pi^*:{\rm Coh}(G/H)\xrightarrow{1:1}\C,\,\,\,U\mapsto \left(\pi^*U,\left(\rho^{\psi^{-1}}_U\ot\id\right)\rho^{\psi}_U\right).$$
\item
Set $\mathbf{F}:=\mathbf{F}_H$. We have an equivalence of abelian categories
$${\rm Coh}(G/H)\boxtimes {\rm Rep}_k (H)\xrightarrow{\cong} \C,\,\,\,\,U\boxtimes V\mapsto \pi^*U \ot \mathbf{F}(V).$$
\end{enumerate}
\end{theorem}

\begin{proof}
\footnote{ 
See \cite[Example 5.8]{BG} for a different proof.}
(1) Follows from Theorem \ref{helpful2} (since $H$ is normal).

(2) Since for any $\bar{g}\in (G/H)(k)$, the object $\pi^*\delta_{\bar{g}}\in \C_{gH}$ is invertible, we have an equivalence of abelian categories
$$\pi^*\delta_{\bar{g}}\ot - :\overline{\C_{H}}\xrightarrow{\cong} \overline{\C_{gH}},\,\,\,S\mapsto \pi^*\delta_{\bar{g}}\ot S.$$

Now since we have an equivalence
$$ 
{\rm Coh}(G/H)_{\bar{1}}\boxtimes {\rm Rep}_k (H)\xrightarrow{\cong}\overline{\C_{H}},\,\,\,U\boxtimes V\mapsto \pi^*U\ot \mathbf{F}(V),$$
it follows that for any $\bar{g}\in (G/H)(k)$, the functor 
$${\rm Coh}(G/H)_{\bar{g}}\boxtimes {\rm Rep}_k (H)\xrightarrow{\cong} \overline{\C_{gH}},\,\,\,U\boxtimes V\mapsto \pi^*U\ot \mathbf{F}(V),$$
is an equivalence of abelian categories, 
which implies the statement.
\end{proof}

\subsection{The connected case}\label{sec:The connected case}
Assume that $G=G^{\circ}$. Recall the group scheme embedding $\partial:=\partial_1:H\xrightarrow{1:1} H\times H$, $h\mapsto (h,h)$ (\ref{thetag}). It is clear that $\partial^{\sharp}:\O(H)^{\ot 2}\to \O(H)$ is the multiplication map of $\O(H)$.

Note that the map
$$\widetilde{\mathfrak{c}}:\O(H)\to \O(H\times H),\,\,\,f\mapsto 1\ot f,$$
is a cleaving map (\ref{nbpgammatilde}) with convolution inverse
$$\widetilde{\mathfrak{c}}^{-1}:\O(H)\to \O(H\times H),\,\,\,f\mapsto 1\ot {\rm S}(f).$$
Since the induced $2$-cocycle $\widetilde{\sigma}$ is trivial, it follows that   
\begin{gather*} 
\widetilde{\phi}:\O\left(H\backslash \left(H\times H\right)\right)\ot \O(H)\xrightarrow{\cong}\O(H\times H),\,\,\,{\rm f}\ot f\mapsto f'(1\ot f),\\
\widetilde{\phi}^{-1}:\O(H\times H)\xrightarrow{\cong} \O\left(H\backslash \left(H\times H\right)\right)\ot \O(H),\,\,\,{\rm f}\mapsto {\rm f}_1\widetilde{\mathfrak{c}}^{-1}(\partial^{\sharp}({\rm f}_2))\ot \partial^{\sharp}({\rm f}_3),
\end{gather*}
and $\widetilde{\alpha}:\O(H\times H)\twoheadrightarrow \O\left(H\backslash \left(H\times H\right)\right)$, ${\rm f}\mapsto {\rm f}_1\widetilde{\mathfrak{c}}^{-1}(\partial^{\sharp}({\rm f}_2))$ (see \S\ref{sec:Right principle homogeneous spaces}).
%

For any $V\in{\rm Rep }_k(H)$, let ${\rm F}_V:=\widetilde{\alpha}\ot\id$ (see \S\ref{sec:thesecondequivalence}). We have 
\begin{gather*}
{\rm F}_V:\O(H\times H)\ot^{\O(H)}V\xrightarrow{\cong}\O(H\backslash (H\times H))\ot_k V,\\
{\rm f}\ot v\mapsto {\rm f}_1\widetilde{\mathfrak{c}}^{-1}(\partial^{\sharp}({\rm f}_2))\ot v,
\end{gather*}
and
\begin{gather*}
{\rm F}_V^{-1}:\O(H\backslash (H\times H))\ot_k V\xrightarrow{\cong}\O(H\times H)\ot^{\O(H)}V,\\
{\rm f}\ot v\mapsto {\rm f}\widetilde{\mathfrak{c}}(v^{1})\ot v^{0}.
\end{gather*}
Recall also (\ref{nuviexp}) the map 
\begin{equation*}
\rho_V:\O(H\backslash (H\times H))\ot_k V\to \O(H\backslash (H\times H))\ot_k V\ot \O(H\times H)_{\psi^{-1}\times\psi}.
\end{equation*}

\begin{lemma}\label{deflambdaV11}
For any ${\rm f}\ot v\in\O(H\backslash (H\times H))\ot_k V$, we have
\begin{equation*}
\rho_V({\rm f}\ot v)={\rm f}_1\left(\psi_1^{-1}\ot {\rm S}(\psi_2^{-1}\partial^{\sharp}\left({\rm f}_2\right))\right)\ot v^0\ot (\psi^{-2}\ot v^1){\rm f}_3. 
\end{equation*}
\end{lemma}

\begin{proof}
Set $\Psi:=\psi^{-1}\times \psi$. By (\ref{nuviexp}), we have
\begin{eqnarray*}
\lefteqn{\rho_V\left({\rm f}\ot v\right)}\\
& = & \Psi_1^1{\rm f}_1\widetilde{\mathfrak{c}}(v^1)_1\widetilde{\mathfrak{c}}^{-1}\partial^{\sharp}\left(\Psi_2^1{\rm f}_2\widetilde{\mathfrak{c}}(v^1)_2\right)\ot v^0\ot \Psi^2{\rm f}_3\widetilde{\mathfrak{c}}(v^1)_3\\
& = & \Psi_1^1{\rm f}_1(1\ot v^1_1)\left(1\ot {\rm S}\partial^{\sharp}\left(\Psi_2^1{\rm f}_2(1\ot v^1_2)\right)\right)\ot v^0\ot \Psi^2{\rm f}_3(1\ot v^1_3)\\
& = & \Psi_1^1{\rm f}_1\left(1\ot v^1_1{\rm S}(v^1_2)\right)\left(1\ot {\rm S}\partial^{\sharp}\left(\Psi_2^1{\rm f}_2\right)\right)\ot v^0\ot \Psi^2{\rm f}_3(1\ot v^1_3)\\
& = & \Psi_1^1{\rm f}_1\left(1\ot {\rm S}\partial^{\sharp}\left(\Psi_2^1{\rm f}_2\right)\right)\ot v^0\ot \Psi^2{\rm f}_3(1\ot v^1)\\
& = & (\psi_1^{-1}\ot \psi_1^{1}){\rm f}_1\left(1\ot {\rm S}\partial^{\sharp}\left((\psi_2^{-1}\ot \psi_2^{1}){\rm f}_2\right)\right)\ot v^0\ot (\psi^{-2}\ot \psi^{2}){\rm f}_3(1\ot v^1)\\
& = & (\psi_1^{-1}\ot \psi_1^{1}){\rm f}_1\left(1\ot {\rm S}(\psi_2^{-1}\psi_2^{1}){\rm S}\partial^{\sharp}\left({\rm f}_2\right)\right)\ot v^0\ot (\psi^{-2}\ot \psi^{2}){\rm f}_3(1\ot v^1)\\
& = & {\rm f}_1\left(\psi_1^{-1}\ot \psi_1^{1}{\rm S}(\psi_2^{-1}\psi_2^{1}){\rm S}\partial^{\sharp}\left({\rm f}_2\right)\right)\ot v^0\ot (\psi^{-2}\ot \psi^{2}){\rm f}_3(1\ot v^1)\\
& = & {\rm f}_1\left(\psi_1^{-1}\ot {\rm S}(\psi_2^{-1}){\rm S}\partial^{\sharp}\left({\rm f}_2\right)\right)\ot v^0\ot (\psi^{-2}\ot 1){\rm f}_3(1\ot v^1)\\
& = & {\rm f}_1\left(\psi_1^{-1}\ot {\rm S}(\psi_2^{-1}\partial^{\sharp}\left({\rm f}_2\right))\right)\ot v^0\ot (\psi^{-2}\ot v^1){\rm f}_3,
\end{eqnarray*}
for any ${\rm f}\ot v\in \O\left(H\backslash H\times H\right)\ot V$, as claimed.
\end{proof}

Now recall the scheme isomorphism $\mathfrak{j}:=\mathfrak{j}_1$ (\ref{doucosetsch}).
We have  
\begin{gather*}
\mathfrak{j}^{\sharp}:\O(H)\xrightarrow{\cong}\O\left(H\backslash \left(H\times H\right)\right),\,\,\,f\mapsto {\rm S}(f_1)\ot f_2,
\end{gather*}
and
\begin{gather*}
(\mathfrak{j}^{\sharp})^{-1}:\O\left(H\backslash \left(H\times H\right)\right)\xrightarrow{\cong}\O(H),\,\,\,{\rm f}\mapsto (\varepsilon\bar{\ot}\id)({\rm f}).
\end{gather*}

\begin{theorem}\label{projcoverconnggst}
The following hold:
\begin{enumerate}
\item
We have $\C=\overline{\C_H}$.
\item
We have an equivalence of tensor categories   
$$\mathbf{F}:=\mathbf{F}_H:\Rep_k(H)\xrightarrow{\cong}\C_H,\,\,\,V\mapsto \iota_*\left(\O(H)\ot_k V,\rho_V^1\right),$$
where for every $f\ot v\in \O\left(H\right)\ot V$, 
$$\rho_V^1\left(f\ot v\right)={\rm S}(\psi^{-1})f_2\ot v^0\ot \psi^{-2}{\rm S}(f_1)\ot v^1f_3.$$ 
\item
For any simple $V\in \Rep_k(H)$, we have   
$$P_{\C}\left(\mathbf{F}(V)\right)\cong 
\left(\O(G)\ot_k P_{H}(V),R_V^1\right),$$
where for every ${\rm f}\ot x\in \O(G)\ot P_{H}(V)$,   
$$R_V^1\left({\rm f}\ot x\right)={\rm f}\ot x^0\ot 1\ot x^1.$$
\item
For any $V\in \Rep_k(H)$, we have
$$\FPdim\left(P_{\C}\left(\mathbf{F}\left(V\right)\right)\right)=\frac{|G|}{|H|}\dim\left(P_{H}\left(V\right)\right).$$
\end{enumerate}
\end{theorem}

\begin{proof}
Follow from Theorem \ref{projsimpobjs}, except for the formulas of $\rho_V^1$ and $R_V^1$. 

(1) By Lemma \ref{deflambdaV11}, for every $f\ot v\in \O(H)\ot_k V$, we have 
\begin{eqnarray*}
\lefteqn{\rho_V\left(\mathfrak{j}_{1}^{\sharp}\ot\id\right)\left(f\ot v\right)=\rho_V\left({\rm S}(f_1)\ot f_2\ot v\right)}\\
& = & \left({\rm S}(f_1)\ot f_2\right)_1\left(\psi_1^{-1}\ot {\rm S}(\psi_2^{-1}\partial^{\sharp}\left(\left({\rm S}(f_1)\ot f_2\right)_2\right))\right)\ot v^0\ot (\psi^{-2}\ot v^1)\left({\rm S}(f_1)\ot f_2\right)_3\\
& = & \left({\rm S}(f_3)\ot f_4\right)\left(\psi_1^{-1}\ot {\rm S}(\psi_2^{-1}\partial^{\sharp}\left(\left({\rm S}(f_2)\ot f_5\right)\right))\right)\ot v^0\ot (\psi^{-2}\ot v^1)\left({\rm S}(f_1)\ot f_6\right)\\
& = & \left({\rm S}(f_3)\ot f_4\right)\left(\psi_1^{-1}\ot {\rm S}(\psi_2^{-1}{\rm S}(f_2)f_5)\right)\ot v^0\ot (\psi^{-2}\ot v^1)\left({\rm S}(f_1)\ot f_6\right)\\
& = & \left({\rm S}(f_3)\ot f_4\right)\left(\psi_1^{-1}\ot {\rm S}(\psi_2^{-1})f_2{\rm S}(f_5)\right)\ot v^0\ot (\psi^{-2}\ot v^1)\left({\rm S}(f_1)\ot f_6\right)\\
& = & \left({\rm S}(f_3)\psi_1^{-1}\ot {\rm S}(\psi_2^{-1})f_2f_4{\rm S}(f_5)\right)\ot v^0\ot (\psi^{-2}\ot v^1)\left({\rm S}(f_1)\ot f_6\right)\\
& = & \left({\rm S}(f_3)\psi_1^{-1}\ot {\rm S}(\psi_2^{-1})f_2\varepsilon(f_4)\right)\ot v^0\ot (\psi^{-2}\ot v^1)\left({\rm S}(f_1)\ot f_5\right)\\
& = & \left({\rm S}(f_3)\psi_1^{-1}\ot {\rm S}(\psi_2^{-1})f_2\right)\ot v^0\ot (\psi^{-2}\ot v^1)\left({\rm S}(f_1)\ot f_4\right)\\
& = & {\rm S}(f_3)\psi_1^{-1}\ot {\rm S}(\psi_2^{-1})f_2\ot v^0\ot \psi^{-2}{\rm S}(f_1)\ot v^1f_4.
\end{eqnarray*}
Thus, by Theorem \ref{prsimmodrep}, we have
\begin{eqnarray*}
\lefteqn{\rho^1_V\left(f\ot v\right)=\left(\left(\mathfrak{j}_{1}^{\sharp}\right)^{-1}\ot\id^{\ot 3}\right)\rho_V\left(\mathfrak{j}_{1}^{\sharp}\ot\id\right)\left(f\ot v\right)}\\
& = & \left(\left(\mathfrak{j}_{1}^{\sharp}\right)^{-1}\ot\id^{\ot 3}\right)\left({\rm S}(f_3)\psi_1^{-1}\ot {\rm S}(\psi_2^{-1})f_2\ot v^0\ot \psi^{-2}{\rm S}(f_1)\ot v^1f_4 \right)\\
& = & \left(\mathfrak{j}_{1}^{\sharp}\right)^{-1}\left({\rm S}(f_3)\psi_1^{-1}\ot {\rm S}(\psi_2^{-1})f_2\right)\ot v^0\ot \psi^{-2}{\rm S}(f_1)\ot v^1f_4\\
& = & \varepsilon\left({\rm S}(f_3)\psi_1^{-1}\right){\rm S}(\psi_2^{-1})f_2\ot v^0\ot \psi^{-2}{\rm S}(f_1)\ot v^1f_4\\
& = & \varepsilon\left({\rm S}(f_3)\right){\rm S}(\psi^{-1})f_2\ot v^0\ot \psi^{-2}{\rm S}(f_1)\ot v^1f_4\\
& = & {\rm S}(\psi^{-1})f_2\ot v^0\ot \psi^{-2}{\rm S}(f_1)\ot v^1f_3,
\end{eqnarray*}
for every $f\ot v\in \O(H)\ot_k V$, as claimed.
 
(2) Set $\theta:=\theta_{H}$ and $\lambda:=\lambda_{H}$ (see Lemma \ref{closedptproj}). We have 
$$\theta:\O(G)\ot\O(H)\twoheadrightarrow \O(G),\,\,\,{\rm f}\ot f\mapsto\varepsilon(f){\rm f},$$
and 
$$\lambda:\O(G)\xrightarrow{1:1}\O(G)\ot\O(H),\,\,\,{\rm f}\mapsto {\rm f}\ot 1.$$
Also, by (1), we have 
$$\rho^1_{P_{H}(V)}\left(1\ot x\right)={\rm S}(\psi^{-1})\ot x^0\ot \psi^{-2}\ot x^1.$$
Thus, by Theorem \ref{prsimmodrep}, 
we have
\begin{eqnarray*}
\lefteqn{R^1_V\left({\rm f}\ot x\right)=\left(\theta\ot\id^{\ot 3}\right)\left(\id\ot\rho^1_{P_{H}(V)}\right)\left(\lambda\ot\id\right)\left({\rm f}\ot x\right)}\\
& = & \left(\theta\ot\id^{\ot 3}\right)\left(\id\ot\rho^1_{P_{H}(V)}\right)\left({\rm f}\ot 1\ot x\right)\\
& = & \left(\theta\ot\id^{\ot 3}\right)\left({\rm f}\ot {\rm S}(\psi^{-1})\ot x^0\ot \psi^{-2}\ot x^1\right)\\
& = & {\rm f}\varepsilon({\rm S}(\psi^{-1}))\ot x^0\ot \psi^{-2}\ot x^1\\
& = & {\rm f}\ot x^0\ot 1\ot x^1,
\end{eqnarray*}
for any ${\rm f}\ot x\in \O\left(G\right)\ot P_{H}(V)$, as claimed.
\end{proof}

\begin{example}\label{restlie case}
Let $\mathfrak{g}$ be a finite dimensional restricted $p$-Lie algebra, and let $\h\subset \mathfrak{g}$ be a restricted $p$-Lie subalgebra. Let $G,H$ be the associated finite group schemes (see Example \ref{repliealg}). Consider the finite tensor category $\C\left(\mathfrak{g},\mathfrak{h}\right):=\C\left(G,H,1\right)$. By Corollary \ref{projcoverconnggst}, we have $\C(\mathfrak{g},\mathfrak{h})=\overline{{\rm Rep}(\mathfrak{h})}$. 

For example, if $\h$ is unipotent then $\C:=\C(\mathfrak{g},\mathfrak{h})$ is a unipotent tensor category, with unique simple object $\mathbf{1}:=\mathbf{F}(k)=\O(H)$, and 
$$P_{\C}(\mathbf{1})\cong \O(G)\ot_k P_{H}(k)
\cong \O(G)\ot_k \O(H)$$
is the free $\O(G)$-module of rank $|H|$, with trivial left $\O(H)$-cocation and right $\O(H)$-cocation $\id\ot\Delta$.   
So, $\C\cong {\rm Coh}(G)$ as abelian categories, but not necessarily as tensor categories (see Theorem \ref{trivialbrpic} below). \qed
\end{example}

\section{Invertible bimodule categories over ${\rm Coh}(G)$}\label{sec:brauerpicard}
Let $\mathscr{A}$ be a finite tensor category, and let $\mathscr{A}^{\rm rev}$ be the category $\mathscr{A}$ with reversed tensor product. Let $\M$ be a left $\mathscr{A}$-module category. Recall that the opposite category $\mathscr{M}^{\rm op}$ is a right $\mathscr{A}$-module category, with the $\mathscr{A}$-action given by 
$$M\ot^{\mathscr{M}^{\rm op}}A:=A^*\ot^{\mathscr{M}}M;\,\,\,A\in \mathscr{A},\,M\in\mathscr{M}.$$

Recall that an exact $\mathscr{A}$-bimodule category $\mathscr{B}$ is called {\em invertible} if there exists an $\mathscr{A}$-bimodule equivalence
$$\mathscr{B}^{\rm op}\boxtimes_{\mathscr{A}}\mathscr{B}\cong \mathscr{A},$$
or equivalently, if the functor
\begin{equation}\label{rmfunctor2}
R_{\mathscr{B}}:\mathscr{A}^{\rm rev}\to \mathscr{A}^{*}_{\mathscr{B}}={\rm Fun}_{\mathscr{A}}(\mathscr{B},\mathscr{B}),\,\,\,A\mapsto -\ot^{\mathscr{B}} A,
\end{equation}
is an equivalence of tensor categories. Here, $\mathscr{A}^{*}_{\mathscr{B}}$ is with respect to the left $\mathscr{A}$-module structure on $\mathscr{B}$, while $-\ot^{\mathscr{B}} A$ is with respect to the right $\mathscr{A}$-module structure on $\mathscr{B}$. Recall also that any invertible exact bimodule category over $\mathscr{A}$ is indecomposable both as left and right module category over $\mathscr{A}$.

Recall \cite{DN} that ${\rm BrPic}(\mathscr{A})$ denotes the group of invertible exact $\mathscr{A}$-bimodule categories,   
${\rm Pic}(\mathscr{Z}(\mathscr{A}))$ denotes the group of one-sided exact $\mathscr{Z}(\mathscr{A})$-bimodule categories, and there exist group isomorphisms 
\begin{equation}\label{brpicisopic}
{\rm Pic}(\mathscr{Z}(\mathscr{A}))\cong {\rm Aut}^{{\rm br}}(\mathscr{Z}(\mathscr{A}))\cong {\rm BrPic}(\mathscr{A}).
\end{equation}
Recall also that for any $\mathscr{A}$-bimodule categories $\mathscr{B}_1$ and $\mathscr{B}_2$, there exists an equivalence of $\mathscr{A}$-bimodule categories
$$\mathscr{B}_1\boxtimes_{\mathscr{A}}\mathscr{B}_2\cong {\rm Fun}_{\mathscr{A}}(\mathscr{B}_1^{\rm op},\mathscr{B}_2).$$

In the following we describe the group ${\rm BrPic}({\rm Coh}(G))$, which we will denote by ${\rm BrPic}(G)$.

\begin{theorem}\label{trivialbrpic0}
There is a one to one correspondence between equivalence classes of pairs $(H,\psi)$ and equivalence classes of indecomposable exact ${\rm Coh}(G)$-bimodule categories, $(H,\psi)\mapsto \mathscr{B}(H,\psi)$, where $H\subset G$ is an abelian normal closed subgroup scheme and $\psi\in Z^2(H,\mathbb{G}_m)$.
\end{theorem}

\begin{proof}
Assume $(H,\psi)$ is as stated, and consider the abelian category $\mathscr{B}:={\rm Coh}(G/H)$. Recall that $(H,\psi)$ determines a structure of a left indecomposable exact ${\rm Coh}(G)$-module category on $\mathscr{B}$, which we denote by $\M(H,\psi)$ (see \S \ref{sec:M(H,psi)}). Similarly, since $H\subset G^{\rm op}$ is a closed subgroup scheme (as $H$ is abelian), $(H,\psi)$ determines a structure of a left indecomposable exact ${\rm Coh}(G^{\rm op})$-module category on the category ${\rm Coh}(G^{\rm op}/H)={\rm Coh}(H\backslash G)=\mathscr{B}$ (as $H$ is normal). Thus, since ${\rm Coh}(G^{\rm op})^{\rm rev}={\rm Coh}(G)$, $(H,\psi)$ determines a structure of a right indecomposable exact ${\rm Coh}(G)$-module category on $\mathscr{B}$, which we denote by $\mathscr{N}(H,\psi)$ (see \S \ref{sec:M(H,psi)}). Thus, $(H,\psi)$ determines  
a canonical structure of an indecomposable exact ${\rm Coh}(G)$-bimodule category on $\mathscr{B}$, which we will denote by $\mathscr{B}(H,\psi)$.

Conversely, let $\mathscr{B}$ be an indecomposable exact ${\rm Coh}(G)$-bimodule category. Then by \cite{G}, $\mathscr{B}=\M(H,\psi)$ as {\em left} ${\rm Coh}(G)$-module categories, and $\mathscr{B}=\mathscr{N}(L,\eta)$ as {\em right} ${\rm Coh}(G)$-module categories. (Note that by Theorem \ref{modomega}, the pairs $(H,\psi)$, $(L,\eta)$ are unique since $G$ is connected.) In particular, since $(\O(H),\Delta_{\psi})$ is the unique simple object of $\M(H,\psi)$, and $(\O(L),\Delta_{\eta})$ is the unique simple object of $\mathscr{N}(L,\eta)$ \cite{G}, it follows that $(H,\psi)=(L,\eta)$ (since $G$ is connected). Thus, ${\rm Coh}(G/H)={\rm Coh}(H\backslash G)$, so $H$ is normal. Moreover, since $\mathscr{N}(H,\eta)$ is a left ${\rm Coh}(G^{\rm op})$-module category, it follows that $H\subset G^{\rm op}$ is a closed subgroup scheme, so $H$ is abelian.

Finally, it is straightforward to verify that the above two assignments are inverse to each other.
\end{proof}

Recall that for a finite abelian group scheme $H,H^D$ denotes the Cartier dual of $H$ (i.e., $\O(H^D)=k[H]$).

\begin{definition}
Let ${\rm IB}(G)$ denote the set of equivalence classes of pairs $(H,\psi)$ such that 
\begin{enumerate}
\item $H\subset G$ is an abelian normal closed subgroup scheme,
\item $\psi\in Z^2(H,\mathbb{G}_m)$ is a $G$-invariant nondegenerate twist, and
\item 
$G^{\rm op}=H^D\rtimes K$, where $K:=G/H$. 
\end{enumerate}
Here, $(H_1,\psi_1)$ is equivalent to $(H_2,\psi_2)$ if $H_1=H_2=H$ and $[\psi_1]=[\psi_2]$ in $H^2(H, \mathbb{G}_m)$ (see Definition \ref{defequivpairs}, Remark \ref{helpful}). Also, $H^D$ is viewed as an abelian normal closed subgroup scheme of $G$ via the $G$-equivariant group scheme isomorphism $f:H^D \xrightarrow{\cong} H$ induced by $\psi\psi_{21}^{-1}$. \qed 
\end{definition}

\begin{theorem}\label{trivialbrpic}
The assignment $(H,\psi)\mapsto \mathscr{B}(H,\psi)$ determines a one to one correspondence between ${\rm IB}(G)$ and ${\rm BrPic}(G)$.
Moreover, if $\mathscr{B}(H,\psi)$ is invertible then $\mathscr{B}(H,\psi)^{-1}=\mathscr{B}(H,\psi^{-1})$.
\end{theorem}

\begin{proof}
Assume $(H,\psi)\in {\rm IB}(G)$. We have an embedding   
\begin{equation}\label{rephcohhd}
f_*:\Rep_k(H)={\rm Coh}(H^D)\xrightarrow{\cong}{\rm Coh}(H)\hookrightarrow{\rm Coh}(G^{\rm op})
\end{equation}
of tensor categories, whose tensor structure is given by $\psi$.

Let $\mathscr{B}:=\mathscr{B}(H,\psi)$, and set $\C:=\C(G,1,H,\psi)$. Recall that 
$$\C\cong {\rm Fun}_{{\rm Coh}(G)}\left(\M\left(H,\psi\right),\M\left(H,\psi\right)\right)$$
as tensor categories.  
We have to show that the tensor functor 
\begin{equation*}
R:=R_{\mathscr{B}}:{\rm Coh}(G^{\rm op})\xrightarrow{} \C,\,\,\,X\mapsto -\,\ot^{\mathscr{N}(H,\psi)} X
\end{equation*}
is an equivalence (see (\ref{rmfunctor2})). 

Since ${\rm Coh}(G^{\rm op})$ and $\C$ have the same FP dimension, it suffices to show that $R$ is surjective. To this end, observe that the restriction of $R$ to $\Rep_k(H)$ (via \ref{rephcohhd}) coincides with $\mathbf{F}$, and the restriction of $R$ to ${\rm Coh}(G/H)$ coincides with $\mathbf{I}$, so the claim follows from Corollary \ref{projcovernormalg}.

Conversely, assume that $\mathscr{B}\in {\rm BrPic}(G)$. Then by Theorem \ref{trivialbrpic0}, $\mathscr{B}=\mathscr{B}(H,\psi)$, where $H\subset G$ is an  abelian normal closed subgroup scheme and $\psi\in Z^2(H,\mathbb{G}_m)$. Set $\C:=\C(G,1,H,\psi)$, so 
$$\C\cong {\rm Fun}_{{\rm Coh}(G)}\left(\M\left(H,\psi\right),\M\left(H,\psi\right)\right)$$
as tensor categories.  
By assumption (see (\ref{rmfunctor2})), we have a tensor equivalence 
\begin{equation}\label{neededequiv}
R:=R_{\mathscr{B}}:{\rm Coh}(G^{\rm op})\xrightarrow{\cong} \C,\,\,\,X\mapsto -\,\ot^{\mathscr{N}(H,\psi)} X.
\end{equation} 
Note that for any $X$ in ${\rm Coh}(H)={\rm Coh}(H^{\rm op})$, we have 
$$R(X)\in {\rm Fun}_{{\rm Coh}(H)}\left(\M\left(H,\psi\right),\M\left(H,\psi\right)\right)=\C_H.$$ Since $\C_H$ has Frobenious-Perron dimension $|H|$, we see that $R$ restricts to a tensor equivalence 
$$R\iota_*:{\rm Coh}(H)\xrightarrow{\cong} \C_H,\,\,\,X\mapsto -\,\ot^{\mathscr{N}(H,\psi)} X.$$
Thus, by Corollary \ref{projcoverconnggst}, we obtain a tensor equivalence
\begin{equation}\label{fiberfunctorL}
T:{\rm Coh}(H)\xrightarrow{R\iota_*}\C_H\xrightarrow{{\bf F}^{-1}} {\rm Rep}(H)={\rm Coh}(H^D).
\end{equation} 
Moreover, since 
the tensor structure on $T$ is given by $\psi$, $T$ is induced from the skew-symmetric bicharacter $\psi \psi_{21}^{-1}$ on $H$. Hence, $\psi$ is $G$-invariant and nondegenerate, i.e.,  
it defines a $G$-equivariant group scheme isomorphism $H^D\xrightarrow{\cong}H$, which allows us to view $H^D$ as an abelian normal closed subgroup scheme of $G$.

Finally, by (\ref{neededequiv}) and Corollary \ref{projcovernormalg}, there is an injective tensor functor 
$$R^{-1}\mathbf{I}:{\rm Coh}(G/H)\xrightarrow{1:1}{\rm Coh}(G^{\rm op}).$$
But since $G$ is connected, this implies that $K:=G/H$ is a closed subgroup scheme of $G^{\rm op}$. Also, since
$$\left(R^{-1}\mathbf{I}\right)\left({\rm Coh}(G/H)\right)\cap \left(R^{-1}\mathbf{F}\right)\left({\rm Coh}(H^D)\right)=\Vect,$$
it follows that $H^D\cap K=1$. This implies that $G^{\rm op}= H^D\rtimes K$. Thus, $(H,\psi)\in {\rm IB}(G)$, as claimed.
\end{proof}

Let $\mathscr{B}(H_1,\psi_1)$ and $\mathscr{B}(H_2,\psi_2)$ be two invertible ${\rm Coh}(G)$-bimodule categories. 
Let $L:=\Delta(H_1\cap H_2)$, $A:=(H_1\times H_2)/L$, $\pi:H_1\times H_2\twoheadrightarrow A$ be the canonical surjective group scheme morphism, and $H:=\pi(L^{\perp})$ the image of the orthogonal complement of $L$ under $\psi_1\times\psi_2$. The latter means that $L^{\perp}\subset H_1\times H_2$ is the largest closed subgroup scheme such that the Hopf algebra map
$$\O(L)^*\xrightarrow{1:1}\O(H_1\times H_2)^*\xrightarrow{(\psi_1\times\psi_2)(\psi_1\times\psi_2)_{21}^{-1}}\O(H_1\times H_2)\twoheadrightarrow \O(L^{\perp})$$
is the trivial map $x\mapsto \varepsilon(x)1$.

Now let $j:L^{\perp}\xrightarrow{1:1}H_1\times H_2$ be the inclusion morphism. We have
\begin{equation}\label{newdefofyetanothertau}
\tau:=j^{\sharp \ot 2}\left(\psi_1\times \psi_2\right)\in Z^2\left(L^{\perp}/\left(L\cap L^{\perp}\right),\mathbb{G}_m\right).
\end{equation}
Since $\pi$ induces a group scheme isomorphism 
\begin{equation}\label{pi induces a group scheme isomorphism}
\pi:L^{\perp}/(L\cap L^{\perp})\xrightarrow{\cong} H,
\end{equation}  
it follows that  
\begin{equation}\label{newdefofyetanotherpsi}
\psi:=\left(\pi^{\sharp \ot 2}\right)^{-1}\left(\tau\right)\in Z^2(H,\mathbb{G}_m)
\end{equation}
is nondegenerate.  

{\bf Assume} from now on that $G$ is {\em abelian}. Then $\mathscr{B}(H,\psi)=\mathscr{M}(H,\psi)$, and $\mathscr{M}(H,\psi)^{{\rm op}}=\M(H,\psi^{-1})$ for any pair $(H,\psi)$.

Also, in this case, we have an injective group scheme morphism
\begin{equation}\label{AisasubgroupofG}
\iota:A=(H_1\times H_2)/L\xrightarrow{1:1}G,\,\,\,\overline{(h_1,h_2)}\mapsto h_1^{-1}h_2,
\end{equation}
so we can view $A$ (hence, $H$) as a closed subgroup scheme of $G$. 
Thus, it follows from the above and Theorem \ref{trivialbrpic0} that $\mathscr{M}(H,\psi)$ is an exact indecomposable ${\rm Coh}(G)$-bimodule category. 

\begin{theorem}\label{trivialbrpictp}
Assume $G$ is abelian. We have
$$\mathscr{M}(H_1,\psi_1)\boxtimes_{{\rm Coh}(G)}\mathscr{M}(H_2,\psi_2)\cong \mathscr{M}(H,\psi)$$
as ${\rm Coh}(G)$-module categories. In particular, $\mathscr{M}(H,\psi)$ is invertible.
\end{theorem}

\begin{proof}
Recall that there is a ${\rm Coh}(G)$-module equivalence 
$$\mathscr{M}(H_1,\psi_1)\boxtimes_{{\rm Coh}(G)}\mathscr{M}(H_2,\psi_2)\cong {\rm Fun}_{{\rm Coh}(G)}(\M(H_1,\psi_1^{-1}),\M(H_2,\psi_2)).$$ 
Thus, by Theorem \ref{simmodrep}, there are ${\rm Coh}(G)$-module equivalences  
\begin{gather*}
\mathscr{M}(H_1,\psi_1)\boxtimes_{{\rm Coh}(G)}\mathscr{M}(H_2,\psi_2)\cong {\rm Coh}^{\left((H_1,\psi_1^{-1}),(H_2,\psi_2)\right)}(G)\\
\cong {\rm Coh}^{\left(H_1\times H_2,\psi_1\times \psi_2\right)}(G).
\end{gather*}

Hence, our job is to show that ${\rm Coh}^{\left(H_1\times H_2,\psi_1\times \psi_2\right)}(G)\cong {\rm Coh}^{(H,\psi)}(G)$ as ${\rm Coh}(G)$-module categories. To this end, consider the functor 
$$F:{\rm Coh}^{\left(H_1\times H_2,\psi_1\times \psi_2\right)}(G)\to {\rm Coh}^{\left(H,\psi\right)}(G),\,\,\,(S,\rho)\mapsto (T,\lambda),$$ 
where $$T:=S^{\left(L^{\perp}/(L\cap L^{\perp}),\tau\right)}\subset S,\,\,\,\,\,\,
\lambda:=\left(\id\ot\pi^{\sharp}\right)^{-1}\circ(\rho_{\mid T}).$$ 
It is clear that $F$ is a functor of ${\rm Coh}(G)$-module categories.

Now by Theorem \ref{simmodrep}, the unique simple object of  
${\rm Coh}^{\left(H_1\times H_2,\psi_1\times \psi_2\right)}(G)$ is $\delta:=\O(A)\ot_k V\cong \O(H_1\times H_2)\ot^{\O(L)_{\psi_1\psi_2}}V$, where $V$ is the unique simple object of 
${\rm Corep}_k(L,\psi_1^{-1}\psi_2^{-1})$. Since $F(\delta)=\delta':=\O(H)$ is the unique simple object of ${\rm Coh}^{\left(H,\psi\right)}(G)$, and every object of ${\rm Coh}^{\left(H,\psi\right)}(G)$ has the form $X\ot \delta'$ for some $X\in {\rm Coh}(G)$ \cite{G}, it follows that $F$ is an equivalence of ${\rm Coh}(G)$-module categories.
\end{proof}

\section{The center of ${\rm Coh}(G)$}\label{sec:twisted-double}
Fix a finite group scheme $G$. Let  
$\Delta\colon G\to G\times G$ be the diagonal morphism, 
$\mathbb{G}:=G\times G$, and $\mathbb{H}:=\Delta(G)$.
	
Let $\mathscr{Z}(G):=\mathscr{Z}({\rm Coh}(G))$ be the center of ${\rm Coh}(G)$.
Recall that objects of $\mathscr{Z}(G)$ are pairs 
$(X,\gamma)$, where $X\in {\rm Coh}(G)$ and
$$\gamma:(-\ot X)\xrightarrow{\cong}(X\ot -)$$
is a natural isomorphism satisfying a certain property, usually known as a half-braiding (see, e.g., \cite[Section 7.13]{DGNO}). 
The center $\mathscr{Z}(G)$ is a finite nondegenerate braided tensor category (see, e.g., \cite[Section 8.6.3]{DGNO}).

Recall that there is a canonical equivalence of tensor categories
\begin{equation}\label{secondequiv}
\mathscr{Z}(G)\xrightarrow{\cong}\left({\rm Coh}(G)\boxtimes {\rm Coh}\left(G\right)^{^{\rm rev}}\right)^*_{{\rm Coh}(G)}
\end{equation}
assigning to a pair $(X,\gamma)$ the functor $X\ot -\colon {\rm Coh}(G)\to {\rm Coh}(G)$, equipped with a module structure coming from $\gamma$ (see, e.g., \cite[Proposition 7.13.8]{DGNO}).

Since there is a canonical equivalence of tensor categories
\begin{equation}\label{thirdequiv}
{\rm Coh}(\mathbb{G})^*_{\M(\mathbb{H},1)}\xrightarrow{\cong} \left({\rm Coh}(G)\boxtimes {\rm Coh}\left(G\right)^{^{\rm rev}}\right)^*_{{\rm Coh}(G)},
\end{equation}
it follows from (\ref{secondequiv})--(\ref{thirdequiv}) that there is a canonical equivalence of tensor categories 
\begin{equation}\label{the center is gstc}
\C\left(\mathbb{G},\mathbb{H},1\right)\cong \mathscr{Z}(G).
\end{equation}

Finally, recall that there is a canonical tensor equivalence 
\begin{equation}\label{the center is gequiv}
{\rm Coh}^{(G)}(G)\cong \mathscr{Z}(G),
\end{equation}
where ${\rm Coh}^{(G)}(G)$ is the category of $G$-equivariant sheaves on $G$ with respect to right conjugation.
Thus, (\ref{the center is gstc})--(\ref{the center is gequiv}) yield a canonical equivalence of tensor categories 
\begin{equation}\label{the center is eqcohtc}
\C\left(\mathbb{G},\mathbb{H},1\right)\cong {\rm Coh}^{(G)}(G).
\end{equation}

In this section, we study the tensor category $\mathscr{Z}(G)$ from two perspectives. First, we use Theorem \ref{projsimpobjs}, and the descriptions of $\mathscr{Z}(G)$ mentioned above, to study the abelian structure of $\mathscr{Z}(G)$. We then use (\ref{G(k)crossedproduct}) and \cite{GNN} to describe $\mathscr{Z}(G)$ as a $G(k)$-equivariantization. Each of these descriptions provides a certain direct sum decomposition of $\mathscr{Z}(G)$, and we end the discussion by establishing a relation between the components coming from the two decompositions. 

\subsection{The structure of $\C\left(\mathbb{G},\mathbb{H},1\right)$}
Let ${\rm Y}:=\mathbb{G}/(\mathbb{H}\times \mathbb{H})$ with respect to the right action $\mu_{\mathbb G \times (\mathbb H \times \mathbb H)}$ \eqref{hkaction}.

\begin{theorem}\label{simpobjscenter}  
Let $\C:=\C\left(\mathbb{G} ,\mathbb{H},1\right)$. The following hold:
\begin{enumerate}
\item
For any ${\rm Z}\in {\rm Y}(k)$ with representative ${\rm g}\in {\rm Z}(k)$, we have an equivalence of abelian categories 
$$\mathbf{F}_{{\rm Z}}:{\rm Rep}_k (\mathbb{H}^{{\rm g}})\xrightarrow{\cong} \mathscr{C}_{{\rm Z}},\,\,\,V\mapsto \iota_{{\rm Z}*}\left(\O\left({\rm Z}\right)\ot_k V,\rho^{{\rm g}}_V\right).$$
In particular, we have a tensor equivalence   
$$\mathbf{F}_{\mathbb{H}}:{\rm Rep}_k (\mathbb{H})\xrightarrow{\cong} \mathscr{C}_{\mathbb{H}},\,\,\,V\mapsto \iota_{\mathbb{H}*}\left(\O\left(\mathbb{H}\right)\ot_k V,\rho^{1}_V\right).$$
\item
There is a bijection between equivalence classes of pairs $({\rm Z},V)$, where ${\rm Z}\in {\rm Y}(k)$ is a closed point with representative ${\rm g}\in {\rm Z}(k)$, and $V\in{\rm Rep}_k (\mathbb{H})$ is simple, and simple objects of $\C$, assigning $({\rm Z},V)$ to $\mathbf{F}_{{\rm Z}}(V)$. Moreover, we have a direct sum decomposition of abelian categories
$$\C=\bigoplus_{{\rm Z}\in {\rm Y}(k)}\overline{\C_{{\rm Z}}},$$
and 
$\overline{\C_{\mathbb{H}}}\subset \C$ is a tensor subcategory.
\item
For any $V\in {\rm Rep}_k (\mathbb{H})$, we have
$
\mathbf{F}_{{\rm Z}}(V)^*\cong \mathbf{F}_{{\rm Z}^{-1}}(V^*)$. 
\item
For any ${\rm Z}\in {\rm Y}(k)$ with representative ${\rm g}\in {\rm Z}(k)$, and $V$ in ${\rm Rep}_k (\mathbb{H})$, we have
$$
{\rm FPdim}\left(\mathbf{F}_{{\rm Z}}(V)\right)=\frac{|\mathbb{H}|}{|\mathbb{H}^{{\rm g}}|}{\rm dim}(V).$$
\item
For any ${\rm Z}\in {\rm Y}(k)$ with representative ${\rm g}\in {\rm Z}(k)$, and simple $V\in {\rm Rep}_k (\mathbb{H})$, we have
$$P_{\C}\left(\mathbf{F}_{{\rm Z}}(V)\right)\cong 
\left(\O(\mathbb{G}^{\circ})\ot\O({\rm Z}(k))\ot_k P_{\mathbb{H}^{{\rm g}}}(V),R^{{\rm g}}_V\right),\,\,\,{\rm and}\,\,\,$$
$${\rm FPdim}\left(P_{\C}\left(\mathbf{F}_{{\rm Z}}\left(V\right)\right)\right)=\frac{|\mathbb{G}^{\circ}||\mathbb{H}(k)|}{|\mathbb{H}^{\circ}||\mathbb{H}^{{\rm g}}(k)|}{\rm dim}\left(P_{\mathbb{H}^{{\rm g}}}\left(V\right)\right).$$
\item
For any ${\rm Z}\in {\rm Y}(k)$, we have
$
{\rm FPdim}\left(\overline{\C_{{\rm Z}}}\right)=|\mathbb{G}^{\circ}||{\rm Z}(k)|$. 
\end{enumerate}
\end{theorem}

\begin{proof}
Follow immediately from Theorem \ref{projsimpobjs}.
\end{proof}

\begin{corollary}\label{centerfib}
Equivalence classes of fiber functors on $\C\left(\mathbb{G} ,\mathbb{H},1\right)$ are classified by equivalence classes of pairs $(\mathbb{K},\eta)$, where
$\mathbb{K}\subset \mathbb{G}$ is a closed subgroup scheme, $\eta\in
Z^2(\mathbb{K},\mathbb{G}_m)$, $\mathbb{K}\mathbb{H}=\mathbb{G}$, and $\xi_1^{-1}\in Z^2(\mathbb{K}\cap \mathbb{H},\mathbb{G}_m)$ is nondegenerate.
\end{corollary}

\begin{proof}
Follows from Corollary \ref{fibfungth}.
\end{proof}

\subsection{The structure of ${\rm Coh}^{(G)}(G)$}  
Let ${\rm C}$ be the finite scheme of conjugacy orbits in $G$. 
Then for any closed point $C\in {\rm C}(k)$, $C\subset G$ is closed and $C(k)\subset G(k)$ is a conjugacy class. Fix a representative $g=g_C\in C(k)$, and let $G_C$ denote the centralizer of $g$ in $G$ (so $G_C(k)$ is the centralizer of $g$ in $G(k)$). 

Note that the map
\begin{equation*}\label{bijectionz}
{\rm C}(k)\to {\rm Y}(k),\,\,\,C_g\mapsto {\rm Z}_{(g,1)},
\end{equation*}
is bijective with inverse given by
$${\rm Y}(k)\to {\rm C}(k),\,\,\,{\rm Z}_{(g_1,g_2)}\mapsto C_{g_1g_2^{-1}}.$$
Also, for any $C\in {\rm C}(k)$ with representative $g\in C(k)$, we have   
$$\mathbb{H}^{(g,1)}=\mathbb{H}\cap (g,1)\mathbb{H}(g^{-1},1)=\Delta(G_C).$$

For any $C\in {\rm C}(k)$ with representative $g\in C(k)$, let $\iota_{g}:G_C\hookrightarrow G$ be the inclusion morphism.

Now {\em choose} a cleaving map (\ref{nbpgammatilde}) $\widetilde{\mathfrak{c}_g}:\O(G_C)\xrightarrow{1:1} \O(G)$, and let
\begin{equation}\label{cCalphaCdouble1}
\widetilde{\alpha_g}:\O(G)\twoheadrightarrow\O(G_C\backslash G),\,\,\,f\mapsto f_1\widetilde{\mathfrak{c}_g}^{-1}\left(\iota_{g}^{\sharp}\left(f_2\right) \right)
\end{equation}
(see (\ref{nbpalphatilde0})).
Consider also the split exact sequence of schemes
$$1\to C^{\circ}\xrightarrow{i_{C^{\circ}}} C \mathrel{\mathop{\rightleftarrows}^{\pi_{C}}_{q_{C}}} C(k)\to 1$$ induced 
from (\ref{ses0}), and define the $\O(G)$-linear algebra maps 
\begin{gather*}
\theta_{C}:=\id\ot q_{C}^{\sharp}:\O(G^{\circ})\ot\O(C)\twoheadrightarrow  \O(G^{\circ})\ot \O(C(k)),\,\,\,{\rm and}\\
\lambda_{C}:=\id\ot \pi_{C}^{\sharp}:\O(G^{\circ})\ot \O(C(k))\xrightarrow{1:1}\O(G^{\circ})\ot\O(C).
\end{gather*}

\begin{theorem}\label{simpobjsnew}  
Set $\mathscr{Z}:={\rm Coh}^{(G)}(G)$. The following hold:
\begin{enumerate}
\item
For any $C\in {\rm C}(k)$ with representative $g\in C(k)$, we have an equivalence of abelian categories 
$$\mathbf{F}_{C}:\Rep_k(G_C)\xrightarrow{\cong}\mathscr{Z}_{C},\,\,\,V\mapsto \iota_{C*}\left(\O(C)\ot_k V,\rho_V^g\right),$$
where $\rho_V^g:\O(C)\ot_k V\to \O(C)\ot_k V\ot \O(G)$ is given by
\begin{equation*}
\rho_V^g(f\ot v)= 
\left(\mathfrak{j}_g^{-1}\right)^{\sharp}\widetilde{\alpha_{g}}\left(\mathfrak{j}_g^{\sharp}(f)_1\widetilde{\mathfrak{c}_g}\left(v^{1}\right)_1 \right)\ot
v^0\ot \mathfrak{j}_g^{\sharp}(f)_2\widetilde{\mathfrak{c}_g}\left(v^{1}\right)_2.
\end{equation*}
(Here, $\mathfrak{j}_g:G_C\backslash G\xrightarrow{\cong} C$ is the canonical scheme isomorphism.)

In particular, $\mathbf{F}_{1}:\Rep_k(G)\xrightarrow{\cong} \mathscr{Z}_{1}\hookrightarrow \mathscr{Z}$ coincides with the canonical embedding of braided tensor categories.
\item
There is a bijection between equivalence classes of pairs $(C,V)$, where $C\in {\rm C}(k)$ with representative $g\in C(k)$, and $V\in \Rep_k(G_C)$ is simple, and simple objects of $\mathscr{Z}$, assigning $(C,V)$ to $\mathbf{F}_{C}(V)$. Moreover, we have a decomposition of abelian categories
$$\mathscr{Z}=\bigoplus_{C\in {\rm C}(k)}\overline{\mathscr{Z}_{C}},$$
and
$\overline{\mathscr{Z}_{1}}\subset \mathscr{Z}$ is a tensor subcategory.
\item
For any $V\in \Rep_k(G_C)$, we have
$
\mathbf{F}_C(V)^*\cong \mathbf{F}_{C^{-1}}(V^*)$.
\item
For any $V\in \Rep_k(G_C)$, we have
$$
{\rm FPdim}\left(\mathbf{F}_{C}(V)\right)=\frac{|G|}{|G_C|}{\rm dim}(V)=|C|{\rm dim}(V).$$
\item
For any simple $V\in \Rep_k(G_C)$, we have
$$P_{\mathscr{Z}}\left(\mathbf{F}_{C}(V)\right)\cong 
\left(\O(G^{\circ})\ot\O(C(k))\ot_k P_{G_C}(V),R_V^g\right),$$
where $\O(G)$ acts diagonally, and    
$$R_{V}^g:=\left(\theta_C\ot\id^{\ot 2}\right)\left(\id_{\O(G^{\circ})}\ot\rho_{P_{G_C}(V)}^g\right)\left(\lambda_{C}\ot\id\right).$$ 
In particular, 
$${\rm FPdim}\left(P_{\mathscr{Z}}\left(\mathbf{F}_{C}(V)\right)\right)=\frac{|G|}{|G_C(k)|}{\rm dim}\left(P_{G_C}\left(V\right)\right).$$
\item
For any $C\in {\rm C}(k)$, we have 
$
{\rm FPdim}\left(\overline{\mathscr{Z}_{C}}\right)=\frac{|G|^2}{|G_C(k)|}$. \qed
\end{enumerate}
\end{theorem}

\begin{proof}
Using the preceding remarks, it is straightforward to verify that Theorem \ref{simpobjscenter} translates to the theorem via the equivalence (\ref{the center is eqcohtc}).
\end{proof}

\begin{example}\label{restlie case double}
Assume $G$ is connected (e.g., $G$ is the finite group scheme associated to 
a finite dimensional restricted $p$-Lie algebra $\mathfrak{g}$ as in Example \ref{repliealg}), and consider the finite braided tensor category $\mathscr{Z}:={\rm Coh}^{(G)}(G)$. By Theorem \ref{simpobjsnew}, we have $\mathscr{Z}=\overline{{\rm Rep}_k(G)}$, and for any $(V,r)\in \Rep_k(G)$, we have $\rho^1_V=r:V\to V\ot \O(G)$ . Moreover, for any 
simple $V\in \Rep_k(G)$, we have
$$P_{\mathscr{Z}}\left(\mathbf{F}(V)\right)\cong 
\left(\O(G)\ot_k P_{G}(V),R_V^1\right),$$
where $\O(G)$ acts on the first factor, and  
$$R_{V}^1:=\left(\theta_1\ot\id^{\ot 2}\right)\left(\id_{\O(G)}\ot\rho^1_{P_{G}(V)}\right)\left(\lambda_{1}\ot\id\right).$$ 
Now since  
\begin{gather*}
\theta_{1}:\O(G)\xrightarrow{1:1}\O(G)\ot\O(G),\,\,\,f\mapsto f\ot 1,\\
\lambda_{1}:\O(G)\ot\O(G)\twoheadrightarrow \O(G),\,\,\,f\ot f'\mapsto f\varepsilon(f'),\\
\widetilde{\mathfrak{c}_1}=\id,\,\,\,{\rm and}\,\,\,\widetilde{\alpha_1}=\varepsilon,
\end{gather*}
we see that $R_{V}^1=\id_{\O(G)}\ot\rho_{P_{G}(V)}$, so
$$P_{\mathscr{Z}}\left(\mathbf{F}(V)\right)\cong 
\left(\O(G)\ot_k P_{G}(V),\id\ot\rho_{P_G(V)}\right).$$
\end{example}

\subsection{Short exact sequence of centers}
Recall (\ref{ses0}) the split exact sequence of group schemes
$$1\to G^{\circ}\xrightarrow{i} G \mathrel{\mathop{\rightleftarrows}^{\pi}_{q}} G(k)\to 1.$$ 
By Proposition \ref{equivmorphs}, it induces the tensor functors
$$i_*:{\rm Coh}(G^{\circ})\xrightarrow{1:1} {\rm Coh}(G),\,\,\,i^*:{\rm Coh}(G)\twoheadrightarrow {\rm Coh}(G^{\circ}),$$
$$\pi_*:{\rm Coh}(G)\twoheadrightarrow{\rm Coh}(G(k)),\,\,\,q_*:{\rm Coh}(G(k))\xrightarrow{1:1} {\rm Coh}(G).$$

\begin{theorem}\label{Short exact sequence of centers}
The following hold:
\begin{enumerate}
\item
The functor $\pi_*$ lifts (using $q$) to a surjective tensor functor 
$$\pi_*:{\rm Coh}^{(G)}(G)\twoheadrightarrow {\rm Coh}^{(G(k))}(G(k)).$$
\item
The functor $q_*$ lifts (using $\pi$) to an injective tensor functor 
$$q_*:{\rm Coh}^{(G(k))}(G(k))\xrightarrow{1:1} {\rm Coh}^{(G)}(G).$$
\item
$\pi_*q_*=\id$ (as abelian functors).
\item
The functor $i^*$ lifts to a surjective tensor functor
$$i^*:{\rm Coh}^{(G)}(G)\twoheadrightarrow {\rm Coh}^{(G^{\circ})}(G^{\circ}).$$
\item
$i^*q_*:{\rm Coh}^{(G(k))}(G(k))\to\Vect=\langle \mathbf{1}\rangle \subset {\rm Coh}^{(G^{\circ})}(G^{\circ})$ is the forgetful functor (as abelian functors).
\item
The identity functor ${\rm Coh}(G)\to {\rm Coh}(G)$  lifts (using $i$) to a surjective tensor functor
$${\rm Coh}^{(G)}(G)\twoheadrightarrow {\rm Coh}^{(G^{\circ})}(G).$$
\end{enumerate}
\end{theorem}

\begin{proof}
(1) Take $(V,r)$ in ${\rm Coh}^{(G)}(G)$. Namely, $V$ is an $\O(G)$-module and $r:V\to V\ot \O(G)$ is an $\O(G)$-coaction with respect to right conjugation. By definition, $\pi_*V=V$ is an $\O(G(k))$-module via $\pi^{\sharp}:\O(G(k))\xrightarrow{1:1} \O(G)$. Thus, the $\O(G(k))$-coaction
$$(\id\ot q^{\sharp})r:\pi_*V\to \pi_*V\ot \O(G(k))$$
endows $\pi_*V$ with a structure of an object in ${\rm Coh}^{(G(k))}(G(k))$.

(2) Take $(V,r)$ in ${\rm Coh}^{(G(k))}(G(k))$. Namely, $V$ is an $\O(G(k))$-module and $r:V\to V\ot \O(G(k))$ is an $\O(G(k))$-coaction with respect to right conjugation. Now by definition, $q_*V=V$ is an $\O(G)$-module via $q^{\sharp}:\O(G)\twoheadrightarrow \O(G(k))$. Thus, the $\O(G)$-coaction
$$(\id\ot \pi^{\sharp})r:q_*V\to q_*V\ot \O(G)$$
endows $q_*V$ with a structure of an object in ${\rm Coh}^{(G)}(G)$.

(3) Follows from $\pi_*q_*=(\pi q)_*=\id$.

(4) Take $(V,r)$ in ${\rm Coh}^{(G)}(G)$. By Proposition \ref{equivmorphs}, the $\O(G^{\circ})$-module $i^*V=\O\left(G^{\circ}\right)\ot_{\O\left(G\right)} V$ is equipped with the $\O(G)$-coaction given by $\mu_{G^{\circ}\times G}^{\sharp}\bar{\ot}r$ (where $\mu_{G^{\circ}\times G}:G^{\circ}\times G\to G^{\circ}$ is the right conjugation action of $G$ on $G^{\circ}$). 
Thus, the map 
$$\left(\id\ot\id\ot i^{\sharp}\right)\left(\mu_{G^{\circ}\times G}^{\sharp}\bar{\ot}r\right),$$ 
endows $i^*V$ with a structure of an object in ${\rm Coh}^{(G^{\circ})}(G^{\circ})$.

(5) Take $(V,r)$ in ${\rm Coh}^{(G(k))}(G(k))$. By (2), 
$$q_*(V,r)=(q_*V,(\id\ot \pi^{\sharp})r)\in {\rm Coh}^{(G)}(G),$$ where $q_*V=V$ is an $\O(G)$-module via $q^{\sharp}$. Thus by (4), 
\begin{eqnarray*}
\lefteqn{i^*q_*(V,r)=i^*\left(q_*V,\left(\id\ot \pi^{\sharp}\right)r\right)}\\
& = & \left(\O\left(G^{\circ}\right)\ot_{\O\left(G\right)} q_*V,\left(\id\ot\id\ot i^{\sharp}\right)\left(\mu_{G^{\circ}\times G}^{\sharp}\bar{\ot}\left(\id\ot \pi^{\sharp}\right)r\right)\right)\\
& = & \left(\O\left(G^{\circ}\right)\ot_{\O\left(G\right)} q_*V,\mu_{G^{\circ}\times G^{\circ}}^{\sharp}\bar{\ot}V_{{\rm tr}}\right)
\end{eqnarray*}
(where $\mu_{G^{\circ}\times G^{\circ}}:G^{\circ}\times G^{\circ}\to G^{\circ}$ is the right conjugation action of $G^{\circ}$ on itself). In other words, $i^*q_*$ sends $(V,r)$ to the direct sum of $\dim(V)$ copies of the identity object $\left(\O(G^{\circ}),\mu_{G^{\circ}\times G^{\circ}}^{\sharp}\right)\in{\rm Coh}^{(G^{\circ})}(G^{\circ})$.

(6) Similar. 
\end{proof}

\subsection{$\mathscr{Z}(G)$ as $G(k)$-equivariantization}
Set $\mathscr{D}:={\rm Coh}(G)$, and let $\mathscr{D}^{\circ}:={\rm Coh}(G^{\circ})$. By (\ref{G(k)crossedproduct}), we have 
$$\mathscr{D}\cong\mathscr{D}^{\circ}\rtimes G(k)\cong \bigoplus_{a\in G(k)}\mathscr{D}^{\circ}\boxtimes a$$
as tensor categories, where the associativity constraint on the right hand side category is given by $\omega$ in the obvious way.

Let $\M$ be any $\mathscr{D}^{\circ}$-bimodule
category. Recall \cite{GNN} that the relative center $\mathscr{Z}_{\mathscr{D}^{\circ}}\left(\M\right)$ is the abelian category whose objects are pairs $(M,{\rm c})$, where $M\in \M$ and
\begin{equation}
\label{gamma}
{\rm c} = \{{\rm c}_{X}:X\ot M \xrightarrow{\cong} M\ot X\mid X \in \mathscr{D}^{\circ}\}
\end{equation}
is a natural family of isomorphisms satisfying some compatibility conditions. In particular, the relative center $\mathscr{Z}_{\mathscr{D}^{\circ}}\left(\mathscr{D}\right)$ is a finite $G(k)$-crossed braided tensor category \cite[Theorem 3.3]{GNN}. The $G(k)$-grading on $\mathscr{Z}_{\mathscr{D}^{\circ}}\left(\mathscr{D}\right)$ is given by
$$\mathscr{Z}_{\mathscr{D}^{\circ}}\left(\mathscr{D}\right)=\bigoplus_{a\in G(k)}\mathscr{Z}_{\mathscr{D}^{\circ}}\left(\mathscr{D}^{\circ}\boxtimes a\right),$$
and the action of $G(k)$ 
on $\mathscr{Z}_{\mathscr{D}^{\circ}}(\mathscr{D})$, $h \mapsto \widetilde{T}_h$, is induced from the action of $G(k)$ on $\mathscr{D}^{\circ}$, $h \mapsto T_h$, in the following way. For any $h\in G(k)$, $X\in \mathscr{D}^{\circ}$, and $(Y\boxtimes a,{\rm c})\in \mathscr{Z}_{\mathscr{D}^{\circ}}\left(\mathscr{D}^{\circ}\boxtimes a\right)$, we have an
isomorphism
\begin{equation}
\label{action crpr} \widetilde{{\rm c}}_X:=(T_h\ot T_h){\rm c}_{T_{h^{-1}}(X)}: X \ot T_h(Y)
\xrightarrow{\cong}T_h(Y)\ot T_{hah^{-1}}(X).
\end{equation}
Set 
$$\widetilde{T}_h(Y\boxtimes a,{\rm c}):=(T_{h}(Y)\boxtimes
hah^{-1},\,\widetilde{{\rm c}}).$$ Then $\widetilde{T}_h$ maps $\mathscr{Z}_{\mathscr{D}^{\circ}}(\mathscr{D}^{\circ}\boxtimes a)$
to $\mathscr{Z}_{\mathscr{D}^{\circ}}(\mathscr{D}^{\circ}\boxtimes {hah^{-1}})$.

Note that $\mathscr{Z}_{\mathscr{D}^{\circ}}(\mathscr{D})\cong {\rm Coh}^{(G^{\circ})}(G)$ as tensor categories, and the obvious forgetful tensor functor
\begin{equation}\label{obvious forgetful tensor functor}
\mathscr{Z}(G)=\mathscr{Z}_{\mathscr{D}}\left(\mathscr{D}\right)\to \mathscr{Z}_{\mathscr{D}^{\circ}}\left(\mathscr{D}\right),\,\,\,(X,{\rm c})\mapsto (X,{\rm c}_{\mid \mathscr{D}^{\circ}})
\end{equation}
coincides with the surjective tensor functor given in  Theorem \ref{Short exact sequence of centers}(6).

By \cite[Theorem 3.5]{GNN}, there is an equivalence of tensor categories
\begin{equation}\label{gnntheorem}
F:\mathscr{Z}(G)\xrightarrow{\cong}\mathscr{Z}_{\mathscr{D}^{\circ}}\left(\mathscr{D}\right)^{G(k)}=\left(\bigoplus_{a\in G(k)}\mathscr{Z}_{\mathscr{D}^{\circ}}\left(\mathscr{D}^{\circ}\boxtimes a\right)\right)^{G(k)}.
\end{equation}

For any $C\in {\rm C}(k)$, set
$$\mathscr{E}_C:=\bigoplus_{a\in C(k)}\mathscr{Z}_{\mathscr{D}^{\circ}}\left(\mathscr{D}^{\circ}\boxtimes a\right)^{G(k)}\subset \mathscr{Z}_{\mathscr{D}^{\circ}}\left(\mathscr{D}\right)^{G(k)}.$$
By (\ref{gnntheorem}), $\mathscr{E}_C$ is a Serre subcategory of $\mathscr{Z}_{\mathscr{D}^{\circ}}\left(\mathscr{D}\right)^{G(k)}$, and we have a tensor equivalence 
\begin{equation}\label{gnntheorem11}
F:\mathscr{Z}(G)\xrightarrow{\cong}\bigoplus_{C\in {\rm C}(k)}\mathscr{E}_C.
\end{equation}

\begin{theorem}\label{compdeco}
For any $C\in {\rm C}(k)$, the functor $F$ (\ref{gnntheorem11}) restricts to an equivalence of abelian categories
$$F_C:\overline{\mathscr{Z}(G)_{C}}\xrightarrow{\cong}\mathscr{E}_C=\bigoplus_{a\in C(k)}\mathscr{Z}_{\mathscr{D}^{\circ}}\left(\mathscr{D}^{\circ}\boxtimes a\right)^{G(k)}.$$

In particular, $F$ restricts to an equivalence of tensor categories $$F_1:\overline{\Rep(G)}\xrightarrow{\cong}\mathscr{Z}\left(G^{\circ}\right)^{G(k)}.$$
\end{theorem}

\begin{proof}
Fix $C\in {\rm C}(k)$ with representative $g\in C(k)$. To see that $F\left(\overline{\mathscr{Z}(G)_{C}}\right)\subset \mathscr{E}_C$ it is enough to show that $F\left(\mathscr{Z}(G)_{C}\right)\subset \mathscr{E}_C$ (by the discussion above). To this end, 
it is enough to show that for any simple ${\rm V}\in \mathscr{Z}(G)_{C}$, the simple $G(k)$-equivariant object $F({\rm V})\in\mathscr{Z}_{\mathscr{D}^{\circ}}\left(\mathscr{D}\right)$ is supported on $C$.

So, let ${\rm V}\in \mathscr{Z}(G)_{C}$ be simple. By Theorem \ref{simpobjsnew}, there exists a unique simple $V$ in $\Rep(G_{C})$ such that ${\rm V}={\bf F}_C(V)=\O(C)\ot_{k} V$. It follows that the forgetful image of $F({\rm V})$ in $\mathscr{Z}_{\mathscr{D}^{\circ}}\left(\mathscr{D}\right)$ (\ref{obvious forgetful tensor functor}) lies in $\bigoplus_{a\in C(k)}\mathscr{Z}_{\mathscr{D}^{\circ}}\left(\mathscr{D}^{\circ}\boxtimes a\right)$, so by the proof of \cite[Theorem 3.5]{GNN},  
$F({\rm V})$ is supported on $C$, as claimed.
\end{proof}

\end{document}